\pdfoutput=1
\newif\ifarxiv
\newif\ifau
\arxivtrue

\ifarxiv
\documentclass{amsart}
\usepackage[utf8]{inputenc}
\usepackage[letterpaper,left=1in,top=1in,right=1in,bottom=1in]{geometry}
\usepackage[numbers,square]{natbib}
\fi

\ifau
\documentclass{birkau}
\fi

\usepackage{amsmath}
\usepackage{amssymb}
\usepackage{cancel}
\usepackage{amsfonts}
\usepackage{color}
\usepackage{bm}
\usepackage{booktabs} 
\usepackage{caption} 
\usepackage{subcaption} 
\usepackage{graphicx}
\usepackage{pgfplots}
\pgfplotsset{compat=1.18}
\usepackage[all]{nowidow}
\usepackage[utf8]{inputenc}
\usepackage{xcolor} 
\usepackage{tikz}
\usepackage{tikzscale}
\usetikzlibrary{er,positioning,bayesnet,arrows.meta,bending}
\usepackage{algpseudocode,algorithm,algorithmicx}
\usepackage[bookmarks=false]{hyperref}
\usepackage{cleveref}
\usepackage{url}
\usepackage{enumitem}

\numberwithin{equation}{section}

\theoremstyle{plain}
\newtheorem{theorem}{Theorem}[section]
\newtheorem{lemma}[theorem]{Lemma}

\theoremstyle{definition}
\newtheorem{definition}[theorem]{Definition}
\newtheorem{example}[theorem]{Example}

\newtheorem{claim}{Claim}[theorem]

\newcommand{\id}{\mathrm{id}}
\newcommand{\tuple}[1]{{\bm{#1}}}
\newcommand{\minion}[1]{{\mathcal #1}}
\newcommand{\minAx}[1]{\minion{A}_{#1}}
\newcommand{\minBx}[1]{\minion{B}_{#1}}
\newcommand{\minCx}[1]{\minion{C}_{#1}}
\newcommand{\minDx}[1]{\minion{D}_{#1}}
\newcommand{\minT}{\minion{T}}
\newcommand{\minAk}{\minion{A}_k}
\newcommand{\minBk}{\minion{B}_k}
\newcommand{\minCk}{\minion{C}_k}
\newcommand{\minDk}{\minion{D}_k}
\newcommand{\minCkp}{\minion{C}_{k+1}}
\newcommand{\minAl}{\minion{A}_l}
\newcommand{\minBl}{\minion{B}_l}
\newcommand{\minClp}{\minion{C}_{l+1}}
\newcommand{\minDl}{\minion{D}_l}
\newcommand{\minAinf}{\minion{A}_{\infty}}
\newcommand{\minBinf}{\minion{B}_{\infty}}
\newcommand{\minCinf}{\minion{C}_{\infty}}
\newcommand{\minDinf}{\minion{D}_{\infty}}

\newcommand{\xxxyy}{(\tuple{xxxyy})}

\newcommand{\xxxxy}{(\tuple{xxxxy})}
\newcommand{\yyyyx}{(\tuple{yyyyx})}
\newcommand{\xxyyy}{(\tuple{xxyyy})}

\newcommand{\ess}{\mathrm{ess}}

\newcommand{\clone}[1]{{\mathcal #1}}

\DeclareMathOperator{\Pol}{Pol}
\DeclareMathOperator{\Polid}{IdPol}
\DeclareMathOperator{\Clo}{Clo}
\DeclareMathOperator{\rank}{rank}

\newcommand{\OR}{\ \mathrm{or}\ }
\newcommand{\AND}{\ \mathrm{and}\ }
\newcommand{\true}{\mathrm{true}}
\newcommand{\false}{\mathrm{false}}

\newcommand{\descript}{\mathtt{D}}

\newcommand{\dd}[1]{\mathtt{#1}}
\newcommand{\idemp}{\mathcal{I}}
\newcommand{\chrd}{\mathrm{chr}(\descript)}

\begin{document}

\title[Multisorted Boolean Clones]{Multisorted Boolean Clones Determined by Binary Relations up to Minion Homomorphisms}

\ifarxiv
\author[L. Barto]{Libor Barto}
\fi
\address{Department of Algebra, Faculty of Mathematics and Physics, Charles University in
Prague, Sokolovská 83, 18675 Prague 8, Czech Republic}
\email{libor.barto@gmail.com}

\author[M. Kapytka]{Maryia Kapytka}
\address{Department of Algebra, Faculty of Mathematics and Physics, Charles University in
Prague, Sokolovská 83, 18675 Prague 8, Czech Republic}
\email{mariagrodno@gmail.com}

\thanks{
 Both authors were funded by the European Union (ERC, POCOCOP, 101071674). Views and opinions expressed are however those of the authors only and do not necessarily reflect those of the
European Union or the European Research Council Executive Agency. Neither the European Union nor the granting authority
can be held responsible for them.
A part of this work appeared as the Master thesis \emph{Minion Cores of Clones} of M. Kapytka, Department of Algebra, Faculty of Mathematics and Physics, Charles University, 2023.}



\begin{abstract}
We describe the ordering of a class of clones by minion homomorphisms, also known as minor preserving maps or height 1 clone homomorphisms. 
The class consists of all clones on finite sets determined by binary relations whose projections to both coordinates have at most two elements. This class can be alternatively described up to minion homomorphisms as the class of multisorted Boolean clones determined by binary relations.
We also introduce and apply the concept of a minion core which provides  canonical representatives for equivalence classes of clones, more generally minions, on finite sets.
\end{abstract}

\maketitle

\section{Introduction}\label{sec:intro}

Two algebras that have the same clones of term operations share many properties. For instance, they have the same invariant relations, in particular, the same subuniverses and congruences. Such algebras can thus be considered essentially equal for many purposes. For this reason, finding all clones is an ultimate classification project in universal algebra. 

The most classical positive result in this direction is Post's classification \cite{Post} of all clones on a two-element set, also called \emph{Boolean clones}. Such clones are naturally ordered by inclusion and the resulting lattice is the well-known Post lattice. Unfortunately, a full classification seems currently out of reach already for clones on a three-element set. A notable partial result is Zhuk's  classification of the so-called self-dual clones on a three-element set  \cite{ZhukSelfdual}. Interestingly, the lattice of these clones has continuum cardinality, yet the full description is quite manageable.

A coarser classification of clones (in some class) may be simpler and can still be sufficient for some purposes. For instance, if our primary interest is in identities (more specifically, Mal'cev conditions) rather than concrete operations in the clone, then we can consider two clones equal if they are \emph{homomorphically equivalent}, that is, there are clone homomorphisms between them in both directions. The equivalence classes of clones are then ordered by the existence of a clone homomorphism. In the resulting partially ordered class, the larger a clone is, the more Mal'cev conditions it satisfies. This ordered class is a lattice, it is isomorphic to the lattice of interpretability types of varieties \cite{Neumann,GarciaTaylor}; more precisely, to its sublattice induced by varieties with no constant symbols.  

The ordered set of clones on small sets \emph{does} become simpler in this regime, for instance, the left and right sides of the Post lattice get identified%
\footnote{We are not aware of an explicit description of Boolean clones up to clone homomorphism in the literature.}.
However, it does not seem to become substantially simpler, for instance, there are still continuum many equivalence classes of self-dual clones on a three-element set~\cite{selfdual}.

A still coarser classification project emerged in connection with fixed-template constraint satisfaction problems in \cite{wonderland}. As before, we consider two clones equivalent if they are homomorphically equivalent, and order the equivalence classes by the existence of a homomorphism. The difference is that we weaken the definition of a homomorphism and only require the preservation of minors (minor of an operation is, roughly, obtained by identifying and permuting variables). Such homomorphisms were called height 1 clone homomorphisms in \cite{wonderland} because they can be equivalently defined as arity-preserving functions that preserve identities with exactly one occurrence of a function symbol on both sides (while clone homomorphisms must preserve all identities). They are also called minor preserving maps or \emph{minion homomorphisms}; we use the latter term in this paper and refer to the partially ordered class as the \emph{minion homomorphism poset}%
\footnote{The whole class forms a distributive lattice by an unpublished result of A. Kazda and M. Moore.}.
We also say that we classify clones \emph{up to minion homomorphisms}.

The poset of clones on small sets now gets significantly simpler. For Boolean clones, it becomes a very simple planar lattice, even though it is still countably infinite \cite{BodirskyVucaj}. A more dramatic collapse happens already on a three-element set, for instance, the continuum-sized lattice of self-dual clones becomes a nice countable lattice~\cite{selfdual}. In fact, it is consistent with the current knowledge that there are only countably many clones on finite sets up to minion homomorphisms.

Recall that clones on a finite set can be described by means of relations~\cite{Bodnarcuk:1969,Geiger:1968}: the set of operations that preserve a given set of relations is always a clone, and every clone on a finite set is of this form. The relational counterpart of minion homomorphisms is primitive positive (pp-) constructions \cite{wonderland}. It follows that the minion homomorphism poset for clones on finite sets is isomorphic to the pp-constructibility poset for finite relational structures. Primitive positive constructions are also often directly used to prove inequalities in the poset \cite{smooth,submaximal,submaximalGraphs}. Even though we  deal with clones explicitly described by relations, we prove the inequalities by directly providing minion homomorphisms.

Our main result describes the subposet of the minion homomorphism poset induced by clones on finite sets that are determined by binary relations whose projections onto both coordinates are at most two-element. This class of clones emerged in an ongoing project of the first author with F. Starke, A. Vucaj, J. Zah\'alka, and D. Zhuk whose goal is to classify the subposet induced by clones on a three-element set determined by binary relations%
\footnote{There are about 2 million such clones as announced in \cite{ZhukMoiseev}, they collapse to few hundreds.}. 
This investigation suggests that binary relations with at most two-element projections onto coordinates play an important role and, as the present paper shows, such clones can be classified up to minion homomorphisms without the restriction to size three. 

It turns out that our class of clones can be equivalently (up to minion homomorphisms) described in a nicer way as the class of multisorted Boolean clones determined by binary relations. Here a \emph{$k$-sorted Boolean clone} is a set of $k$-tuples of Boolean operations with closure properties analogical to standard clones%
\footnote{
There is at least one substantially different multisorted setting in the literature, e.g., the one used in \cite{LPW}. The version we use is a natural choice in some contexts, such as  constraint satisfaction problems with multiple domains.
}. 
Alternatively, a $k$-sorted Boolean clone can be defined as a clone on $\{0,1\}^k$ whose operations act component-wise on $k$-tuples. Multisorted clones appear naturally in clone theory \cite{romov1973lattice} and the theory of constraint satisfaction problems \cite{BulatovMultisorted}. The full classification of multisorted Boolean clones may be challenging (even up to minion homomorphic equivalence) but tractable: Taimanov proved in the 80s (recently published in English \cite{Taimanov1,Taimanov2,Taimanov3}) that there are only countably many such multisorted clones. 

Apart from the classification of a class of clones up to minion homomorphisms, this paper makes a conceptual contribution by introducing the concept of a \emph{minion core}, which is a minion (roughly, a set of operations closed under taking minors) such that each of its endomorphisms is an automorphism. We show that each clone on a finite set, more generally each abstract minion that satisfies a certain finiteness condition, has an essentially unique equivalent minion core. Minion cores thus provide canonical representatives for equivalence classes of clones (or minions) on finite sets. We provide minion cores for the clones in the above class.

The paper is organized as follows. The preliminary~\Cref{sec:prelim} introduces the necessary concepts. In \Cref{sec:cores}, we define minion cores and prove their existence and uniqueness. \Cref{sec:classif} is the main part of this work. We classify the multisorted Boolean clones determined by binary relations up to minion homomorphisms and find their minion cores; the cores and their ordering is shown in \Cref{fig:order0} and later in more detail in \Cref{fig:order}. \Cref{sec:small-projections} shows that this class is equivalent to the class of clones determined by binary relations with at most two-element projections. We provide some concluding remarks in \Cref{sec:conclusion}. 

\begin{figure}[ht]
\begin{center}
\resizebox{4.5cm}{!}{\includegraphics{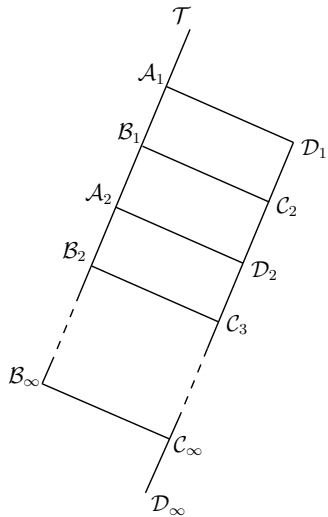}}
\caption{The minion homomorphism poset of multisorted Boolean clones determined by binary relations.} 
\label{fig:order0}
\end{center}
\end{figure}

\section{Clones and minions} \label{sec:prelim}

In this section, we briefly introduce function clones (\Cref{subsec:function-clone}) and function minions (\Cref{subsec:function-minion}), including the multisorted version of these concepts that we work with. We remark that function clones and function minions are often called just clones and minions; this includes the introduction of this paper. We prefer to use the shorter variants for abstract versions of the concepts. 

We then formally introduce (abstract) minions, their homomorphisms, and the induced partially ordered class in \Cref{subsec:minion-homo}. We further make a few observations about homomorphisms between function minions in \Cref{subsec:homo-function-minion}. 

The set of positive integers is denoted by $\mathbb{N}$. For a nonnegative integer $n$ we use the notation
$[n] = \{1,2, \dots, n\}$.
Tuples of elements, functions, or sets are written in boldface and their elements as in, e.g., $\tuple{a} = (a_1, a_2, \dots, a_n)$.
We work with finitary operations of arity at least one. An \emph{$n$-ary operation on $A$} is a function $A^n \to A$. An \emph{operation on $A$} is an $n$-ary operation for some $n \in \mathbb{N}$.

\subsection{Function clones} \label{subsec:function-clone}

A \emph{function clone} on $A$ is a set of operations on $A$ that contains all the projections and is closed under composition. Our notation is as follows.
For a positive integer $n$ and $i \in [n]$, the $n$-ary projection to the $i$th coordinate is denoted $\pi^n_i$, that is,
\[
\pi^n_i (a_1, \dots, a_n) = a_i \mbox{ for every } a_1, \dots, a_n \in A,
\]
where the set $A$ on which $\pi^n_i$ acts should be clear from the context.
For an $n$-ary operation $f$ on $A$ and $m$-ary operations $g_1, \dots, g_n$ on $A$, their composition is denoted by $f(g_1, \dots, g_n)$ or $f \circ (g_1, \dots, g_n)$; we use the latter when we want to emphasize that we are composing.
\[
(f \circ (g_1, \dots, g_n)) (a_1, \dots, a_m) = f (g_1(a_1, \dots, a_m), \dots, g_n(a_1, \dots, a_m))
\]
An $n$-ary operation $f$ on $A$ \emph{preserves} an $m$-ary relation $R$ on $A$ (i.e., a subset of $A^m$ for a positive integer $m$) if $f$ applied component-wise to tuples in $R$ always results in a tuple in $R$. If $\Gamma$ is a set of relations on $A$, an operation $f$ that preserves all relations in $\Gamma$ is called a \emph{polymorphism} of $\Gamma$. The set of all polymorphisms of $\Gamma$ is denoted by $\Pol(\Gamma)$. It is always a function clone and every function clone on a finite set $A$ is of this form: this follows from properties of the Galois correspondence between operations and relations established in~\cite{Bodnarcuk:1969,Geiger:1968}.

A straightforward generalization of function clones to multisorted sets goes as follows.
 By a \emph{$k$-sorted set} we mean a $k$-tuple of sets, e.g., $\tuple{A} = (A_1, A_2, \dots, A_k)$. An \emph{$n$-ary ($k$-sorted) operation} on $(A_1, A_2, \dots, A_k)$ is defined as a $k$-tuple $\tuple{f}=(f_1, f_2, \dots, f_k)$, where $f_i$ is an $n$-ary operation on $A_i$. The \emph{$n$-ary projection on $\tuple{A}$ to the $i$th coordinate} (where $i \in [n]$), denoted by $\tuple{\pi}^n_i$, is the $k$-sorted operation defined by
 $$
 \tuple{\pi}^n_i = (\underbrace{\pi^n_i, \dots, \pi^n_i}_\text{$k$ times}),
 $$
 where the $j$th projection in the tuple is on the set $A_j$.

Let $\tuple f = (f_1, f_2, \dots, f_k)$ be an $n$-ary operation on a $k$-sorted set $\tuple A$ and $\tuple g^1 = (g^1_1, g^1_2, \dots, g^1_k)$, $\tuple g^2= (g^2_1, g^2_2, \dots, g^2_k)$, \dots, $\tuple g^n= (g^n_1, g^n_2, \dots, g^n_k)$ be $m$-ary operations on the same $k$-sorted set $\tuple A$. Then their composition, denoted as $\tuple f \circ (\tuple g^1, \tuple g^2, \dots, \tuple g^n)$, is the $m$-ary operation on $\tuple{A}$ defined component-wise:
$$
\tuple f \circ (\tuple g^1, \tuple g^2, \dots, \tuple g^n) \\
= (f_1 \circ (g^1_1, g^2_1, \dots, g^n_1), \dots, f_k \circ (g^1_k, g^2_k, \dots, g^n_k))
$$

\begin{definition} 
A \emph{function clone} on a $k$-sorted set $\tuple{A}$ is a set of operations on $\tuple{A}$ that contains all the  projections on $\tuple{A}$ and is closed under composition. 
A \emph{Boolean $k$-sorted clone} is a function clone on the \emph{$k$-sorted Boolean set} $(\{0,1\},\{0,1\}, \dots, \{0,1\})$.
\end{definition}
 
An $m$-ary \emph{relation} on a $k$-sorted set $\tuple{A}=(A_1, A_2, \dots, A_k)$ of type $(i_1, \dots, i_m)$, where  $i_1, \dots, i_m \in [k]$, is a subset $R \subseteq A_{i_1} \times A_{i_2} \times \dots \times A_{i_m}$. We say that an $n$-ary operation $\tuple{f} = (f_1, \dots, f_k)$ on $\tuple{A}$ \emph{preserves} such a relation $R$  if the following implication holds. 
$$
\begin{pmatrix} r_{11} \\ r_{12} \\ \vdots \\ r_{1m} \end{pmatrix}, \begin{pmatrix} r_{21} \\ r_{22} \\ \vdots \\ r_{2m} \end{pmatrix}, \dots, \begin{pmatrix} r_{n1} \\ r_{n2} \\ \vdots \\ r_{nm} \end{pmatrix} \in R \implies \begin{pmatrix}
    f_{i_1}(r_{11}, r_{21}, \dots, r_{n1}) \\ f_{i_2}(r_{12}, r_{22}, \dots, r_{n2}) \\ \vdots \\ f_{i_m}(r_{1m}, r_{2m}, \dots, r_{nm})
\end{pmatrix}  \in R
$$
The type of $R$ is regarded as a part of the definition of $R$, so, formally, a $k$-sorted relation is, e.g., a tuple $(R,i_1, \dots, i_m)$. 
We define $\Pol(\Gamma)$ analogously as in the single-sorted setting. As before, $\Pol(\Gamma)$ is a function clone on $\tuple{A}$ and, if all the $A_i$ are finite, every function clone on $\tuple{A}$ is of this form \cite{romov1973lattice}.

Finally, we give a name to the property of relations studied in \Cref{sec:small-projections}.
\begin{definition}
A relation $R \subseteq A_{i_1} \times A_{i_2} \times \dots \times A_{i_m}$ is said to have \emph{small projections} if the projection of $R$ onto every coordinate $j \in [m]$ has size at most two. 
\end{definition}

\subsection{Function minions} \label{subsec:function-minion}

Let $A$ and $B$ be sets and let $n$ be a positive integer. An \emph{$n$-ary operation from $A$ to $B$} is a function from $A^n$ to $B$. 
In general, composition as defined above does not make sense for operations from $A$ to $B$. However, it makes sense to compose such operations with projections on $A$. Explicitly, for an $n$-ary operation $f$ from $A$ to $B$, a positive integer $m$, and $i_1, \dots, i_n \in [m]$, the composition $f \circ (\pi^m_{i_1},\pi^m_{i_2}, \dots, \pi^m_{i_n})$ is the following $m$-ary operation from $A$ to $B$.
$$
(f \circ (\pi^m_{i_1},  \dots, \pi^m_{i_n})) (a_1,  \dots, a_m) = f(a_{i_1},  \dots, a_{i_n})
$$
The operation $f \circ (\pi^m_{i_1},\pi^m_{i_2}, \dots, \pi^m_{i_n})$ is also called a \emph{minor} of $f$. 
A \emph{function minion} on $(A,B)$ is a nonempty set of operations from $A$ to $B$ that is closed under taking minors. The same or similar concept has been also called a minor closed class, clonoid, or minion in the literature. We use the latter for an abstract version of the concept.    

The concept of preserving relations generalizes as follows: An $n$-ary operation $f$ on $A$ \emph{preserves} $(R,S)$, where $R$ is a relation on $A$ and $S$ is a relation on $B$ of the same arity as $R$, if  $f$ applied component-wise to tuples in $R$ always results in a tuple in $S$. The  correspondence between operations and relations in this regime was worked out in \cite{Pippenger}, see also \cite{BG} and \cite{PCSP}.

We generalize function minions to the multisorted setting similarly as we generalized function clones. 
Let $\tuple{A} = (A_1, A_2, \dots, A_k)$ and $\tuple{B}=(B_1,B_2, \dots, B_k)$ be $k$-sorted sets. An \emph{$n$-ary operation from $\tuple{A}$ to $\tuple{B}$} is a $k$-tuple $\tuple{f}=(f_1, f_2, \dots, f_k)$, where $f_i$ is an $n$-ary operation from $A_i$ to $B_i$. 
Note that operations from $\tuple{A}$ to $\tuple{B}$ can still be composed with projections on $\tuple{A}$. Explicitly,
if  $\tuple f = (f_1, f_2, \dots, f_k)$ is a $k$-sorted $n$-ary operation from $\tuple{A}$ to $\tuple{B}$, $m$ is a positive integer, and $i_1, \dots, i_n \in [m]$, then the $j$-th component ($j \in [k]$) of $\tuple f \circ (\tuple{\pi}^m_{i_1}, \dots, \tuple{\pi}^m_{i_m})$ is $f_j \circ (\pi^m_{i_1}, \dots, {\pi}^m_{i_n})$, which is an $m$-ary operation from $A_j$ to $B_j$. As above, the obtained operation from $\tuple{A}$ to $\tuple{B}$ is called a \emph{minor} of $\tuple{f}$. 

We also use an alternative notation for minors. If $\tuple{f}$ is an $n$-ary operation from $\tuple{A}$ to $\tuple{B}$ and $\alpha: [n] \to [m]$, then $\tuple{f}^{(\alpha)}$ is the $m$-ary operation
$$
\tuple{f}^{(\alpha)} = \tuple f \circ (\tuple{\pi}^m_{\alpha(1)}, \tuple{\pi}^m_{\alpha(2)}, \dots, \tuple{\pi}^m_{\alpha(n)}).
$$

 \begin{definition}\label{def:functionminion}
 Let $\tuple{A}$ and $\tuple{B}$ be $k$-sorted sets.
     \emph{A function minion} $\minion{M}$ on $(\tuple{A},\tuple{B})$ is a nonempty set of operations from $\tuple{A}$ to $\tuple{B}$ which is closed under taking minors. A \emph{Boolean $k$-sorted minion} is a function minion on $(\tuple{A},\tuple{A})$, where $\tuple{A}$ is the $k$-sorted Boolean set $(\{0,1\}, \dots, \{0,1\})$.
     
     For a positive integer $n$, we denote by $\minion{M}^{(n)}$ the set of all $n$-ary members of $\minion M$ and call it the \emph{$n$-ary part of $\minion{M}$}.
 \end{definition}

\subsection{Minions and homomorphism} \label{subsec:minion-homo}

A homomorphism between function minions is a mapping that preserves arities and minors.
This concept does not depend on concrete operations in the function minions, it only depends on the mappings $f \mapsto f^{(\alpha)}$. It is therefore appropriate to work with an abstraction of function minions that carries exactly this information, minions. An efficient way to define a minion is as a finitary functor from the category of nonempty sets to itself; homomorphisms are then natural transformations. The following definitions essentially spell this out without the categorical language.

\begin{definition} \label{def:minion}
An (abstract) \textit{minion} $\minion M$ consists of a collection of nonempty sets $(\minion M^{(n)})_{n \in \mathbb{N}}$, together with a \textit{minor map} ${\minion M}^{(\alpha)}:  {\minion M}^{(n)}\to  {\minion M}^{(m)}$ for every function $\alpha: [n]\to [m]$, which satisfies that ${\minion M}^{(\id_{[n]})}= \id_{{\minion M}^{(n)}}$ for all $n \in \mathbb{N}$ and ${\minion M}^{(\alpha)} \circ {\minion M}^{(\beta)} = \minion M^{(\alpha \circ \beta)}$ whenever such a composition makes sense. 
\end{definition} 

Every function minion $\minion{M}$ on a multisorted set is a minion by setting $\minion{M}^{(n)}$ to be the $n$-ary part of $\minion{M}$ (consistently with \Cref{def:functionminion}) and $\minion{M}^{(\alpha)}(\tuple{f}) = \tuple{f}^{(\alpha)}$.

\begin{definition} \label{def:minionhomo}
    Let $\minion M $ and $\minion N$ be minions.
    A \emph{homomorphism} from $\minion M $ to  $\minion N$, written $\xi: \minion M \to \minion N$,  is a collection of functions $(\xi^{(n)}: \minion M^{(n)} \to \minion N^{(n)})_{n \in \mathbb{N}}$ that preserves taking minors, that is, $\xi^{(m)}(\minion M^{(\alpha)}(f))={\minion N}^{(\alpha)}(\xi^{(n)}(f))$ for every $n, m \in \mathbb{N}$, $f \in \minion{M}^{(n)}$, and $\alpha:[n] \to [m]$. We sometimes write $\xi$ instead of $\xi^{(n)}$ if $n$ is clear from the context. 
\end{definition}

We work with homomorphisms in the above sense between function clones. In order to avoid confusion with clone homomorphism (which we do not formally define here), we refer to them as \emph{minion homomorphism}.

Composition of minion homomorphisms, minion isomorphisms, minion automorphisms, etc. are defined in the expected way. Notice that $\xi$ is a minion isomorphism (invertible with respect to composition) if, and only if, each $\xi^{(n)}$ is a bijection. 

We preorder the class of all minions by the existence of minion homomorphisms and define the induced equivalence in the standard way. 

\begin{definition}
    For minions $\minion{M}$ and $\minion{N}$ we write $\minion{M} \leq \minion{N}$ if there exists a minion homomorphism $\minion{M} \to \minion{N}$. We say that minions $\minion{M}$ and $\minion{N}$ are \emph{equivalent}, written $\minion{M} \sim \minion{N}$, if $\minion{M} \leq \minion{N} \leq \minion{M}$.
\end{definition}

We study a part of the partially ordered class whose elements are $\sim$-equivalence classes of minions (in fact, clones) ordered by $\leq$. We informally say that we order minions by minion homomorphisms, although the elements of the partially ordered class are $\sim$-equivalence classes rather than single minions.

The largest minion is, e.g., the unique (one-sorted) function clone $\clone{T}$ on a one-element set.
$$
\clone{T} = \text{ the clone on $\{1\}$}.
$$

The smallest minion is, e.g., the (one-sorted) function clone $\clone{P}$ on the set $\{0,1\}$ containing only projections. 
$$
\clone{P} = \text{ the clone of projections on \{0,1\} }
$$
This minion is isomorphic to the minion given by $\minion{M}^{(n)} = [n]$, $\minion M^{(\alpha)} = \alpha$.

\subsection{Homomorphisms between function minions} \label{subsec:homo-function-minion}

Since our work concerns minion homomorphisms between function minions, we make a few observations in such a context.

For function minions $\minion{M}$, $\minion{N}$ on pairs of multisorted sets $(\tuple A,\tuple B)$ and $(\tuple{C},\tuple{D})$, a collection $\xi=(\xi^{(n)}:\minion M^{(n)} \to \minion{N}^{(n)})_n$ is a minion homomorphism from $\minion{M}$ to $\minion{N}$ if, and only if, $\xi^{(m)}(\tuple f^{(\alpha)}) = (\xi^{(n)}(\tuple{f}))^{(\alpha)}$ for all $m,n \in \mathbb{N}$ and $\alpha:[n] \to [m]$; equivalently,  
        $$
        \xi^{(m)} (\tuple f \circ (\tuple \pi^m_{i_1}, \dots, \tuple \pi^m_{i_n})) = \xi^{(n)} (\tuple f) \circ (\tuple \pi^m_{i_1}, \dots, \tuple \pi^m_{i_n})
        $$
for all $n \in \mathbb{N}$, $\tuple f \in \minion{M}^{(n)}$, $m \in \mathbb{N}$, and $i_1, i_2, \dots, i_n \in [m]$. (Note that the projections on the left hand side are on $\tuple{A}$, while they are on $\tuple{C}$ on the right hand side.)

This can be interpreted as preserving height 1 identities, that is, identities with exactly one function symbol on both sides. For instance, if $\tuple{f} \in \minion{M}$ satisfies the identity $f(x,x,y) \approx f(y,y,x)$, equivalently,  
$$
\tuple{f} \circ (\tuple{\pi}^3_1, \tuple{\pi}^3_1, \tuple{\pi}^3_2) =
\tuple{f} \circ (\tuple{\pi}^3_2, \tuple{\pi}^3_2, \tuple{\pi}^3_1),
$$
then $\tuple{g}=\xi(\tuple{f})$ satisfies the corresponding  identity $g(x,x,y) \approx g(y,y,x)$, equivalently,
$$
\xi(\tuple{f}) \circ (\tuple{\pi}^3_1, \tuple{\pi}^3_1, \tuple{\pi}^3_2) =
\xi(\tuple{f}) \circ (\tuple{\pi}^3_2, \tuple{\pi}^3_2, \tuple{\pi}^3_1).
$$
Indeed, this fact is deduced by applying $\xi$ to both sides of the first displayed equation and using that $\xi$ is a minion homomorphism.

The final observation in particular says that minion homomorphisms between Boolean multisorted minions are fully determined by their binary parts.

\begin{lemma} \label{lem:binaries}
Let $\minion M$ be a minion and $\minion N$ a function minion on a pair of $k$-sorted sets $(\tuple{A},\tuple{B})$. Let $m$ be greater than or equal to the size of each sort $A_i$.
If $\xi,\nu: \minion{M} \to \minion{N}$ are minion homomorphisms such that $\xi^{(m)} = \nu^{(m)}$, then $\xi = \nu$.
\end{lemma}

\begin{proof}
   Let $n \in \mathbb{N}$ and $f \in \minion{M}^{(n)}$. We need to verify that $\xi^{(n)}(f) = \nu^{(n)}(f)$, which we do by checking that for every sort index $i$ and every $\tuple{a} \in A_i^n$, we have $(\xi^{(n)}(f))_i(\tuple{a}) = (\nu^{(n)}(f))_i(\tuple{a})$.
   Take  $\tuple{c} \in A_i^m$ and $\alpha: [n] \to [m]$ such that $c_{\alpha(j)} = a_j$ for every $j \in [n]$, which is possible since $|A_i| \leq m$. Using the definition of minors and the fact that $\xi$ preserves them, we obtain 
   $$
   (\xi^{(n)}(f))_i(\tuple{a})
   =(\xi^{(n)}(f))_i^{(\alpha)}(\tuple c)
   =(\xi^{(m)}(\minion{M}^{(\alpha)}(f)))(\tuple c).
   $$
   The analogous calculation for $\nu$ and the assumption $\xi^{(m)}=\nu^{(m)}$ now finishes the proof.
\end{proof}

\section{Cores} \label{sec:cores}

For every set of relations $\Gamma$ on a finite set $A$, there exists a set of relations $\Delta$ on a finite $B$, the \emph{idempotent core of $\Gamma$}, which is in some sense uniquely determined by $\Gamma$, it contains all the singleton unary relations $\{b\}$ ($b \in B$), and $\Pol(\Gamma) \sim \Pol(\Delta)$. One gain is that $\Pol(\Delta)$ is idempotent and such function clones are somewhat easier to work with. In fact, the aim to identify the function clones $\Pol(\Gamma)$ and $\Pol(\Delta)$ was among the motivations for introducing minion homomorphisms between clones in \cite{wonderland}, called height one (h1) clone homomorphisms therein.
We review the necessary preliminaries in \Cref{subsec:idemp}.

\Cref{subsec:minioncore} introduces the concept of a minion core, which is a natural version of the concept of a core for abstract minions.

\subsection{Idempotent cores} \label{subsec:idemp}

An operation $f \colon A^n \to A$ is \emph{idempotent} if $f(a,a, \dots, a)=a$ for every $a \in A$. In other words, $f$ is idempotent if it preserves all the singleton unary relations $\{a\}$. 
Similarly, an $n$-ary operation $\tuple{f}$ on a multisorted set $\tuple{A}$ is \emph{idempotent} if so are all the components $f_i$. This can be written as $\tuple{f}(\tuple{\pi}^1_1, \dots, \tuple{\pi}^1_1) = \tuple{\pi}^1_1$.  A function clone on $\tuple{A}$ or, more generally, a function minion on $(\tuple{A},\tuple{A})$ is called \emph{idempotent} if so are all of its members.

We denote by $\idemp(\tuple{A})$ the clone of all idempotent operations on $\tuple{A}$. 
For a set $\Gamma$ of  relations on a multisorted set$\tuple{A}$, we write $\Polid(\Gamma)$ for the idempotent part of $\Pol(\Gamma)$:
$$
\Polid(\Gamma) = \Pol(\Gamma) \cap \idemp(\tuple{A})
$$

Every function clone on a finite set is $\sim$-equivalent to an idempotent clone \cite{wonderland}. This fact extends to multisorted clones in a straightforward way. The following formulation will be convenient for our purposes.

\begin{theorem} \label{thm:rel-cores}
  Let $\tuple{A}=(A_1, \dots, A_k)$ be a $k$-sorted set with each $A_i$ finite and let $\Gamma$ be a set of relations on $\tuple{A}$ of arity at most $m \in \mathbb{N}$ such that $\Pol(\Gamma) \not\sim \clone{T}$. Then there exists a $l$-sorted ($l \in \mathbb{N}$) set $\tuple{B}=(B_1, B_2, \dots, B_{l})$ and a set $\Delta$  of relations on $\tuple{B}$ of arity at most $m$ such that
  \begin{itemize}
    \item $\Pol(\Gamma) \sim \Polid(\Delta) = \Pol(\Delta)$  
    \item $1 < |B_i| \leq \max_{j \in \{1, \dots, k\}} |A_j|$ for every $i \in [l]$
    \item If every relation in $\Gamma$ has small projections, then so does every relation in $\Delta$.
  \end{itemize}
\end{theorem}

Here is a brief sketch: take a unary operation in $\Pol(\Gamma)$ with minimal range w.r.t. component-wise inclusion, define $\tuple{B}$ and $\Delta$ as images of $\tuple{A}$ and $\Gamma$ under this operation, add all singleton unary relations of every type, and remove singleton sorts. If no sort remains, then $\Pol(\Gamma) \sim \minion{T}$, otherwise $\Pol(\Gamma) \sim \Pol(\Delta) = \Polid(\Delta)$ by~\cite{wonderland}. 

Next we observe that idempotent minions contain all the projections and minion endomorphisms preserve them.

\begin{lemma} \label{lem:cores-projections}
    Let $\minion M$ and $\minion N$ be idempotent minions on $(\tuple{A},\tuple{A})$ and $(\tuple{B},\tuple{B})$, respectively. Let $\xi: \minion M \to \minion N$ be a minion homomorphism. Then $\minion M$ (resp. $\minion N$) contains all the projections on $\tuple{A}$ (resp. $\tuple{B}$) and for any $m \in \mathbb N$ and $i \in [m]$ we have
    $
    \xi(\tuple{\pi}^m_i) = \tuple{\pi}^m_i.
    $
\end{lemma}

\begin{proof}
 The only idempotent unary multisorted operation on $\tuple{A}$ is $\tuple{\pi}^1_1$. One of its minors is $\tuple{\pi}^1_1 \circ (\tuple{\pi}^m_i) = \tuple{\pi}^m_i$, so it is in $\minion M$ and we have
\[
\xi(\tuple{\pi}^m_i) = \xi(\tuple{\pi}^1_1 \circ (\tuple{\pi}^m_i)) = \xi(\tuple{\pi}^1_1) \circ (\tuple{\pi}^m_i) = \tuple{\pi}^1_1 \circ (\tuple{\pi}^m_i) = \tuple{\pi}^m_i.
\]
\end{proof}

If $\xi: \minion M \to \minion N$ is as in the last lemma, then $\xi$ also preserves identities of height at most 1. For instance if $\tuple{f} \in \minion{M}$ satisfies $f(x,x,y) \approx y$, then the corresponding identity $g(x,x,y) \approx y$ is satisfied by $\tuple{g} = \xi(\tuple f)$.

\subsection{Minion core} \label{subsec:minioncore}

A core of a structure is a homomorphically equivalent structure whose every endomorphism is an automorphism. For abstract minions, this concept takes the following form. 

\begin{definition}
    A minion $\minion{N}$ is called a \emph{minion core} if every minion homomorphism from $\minion{N}$ to itself is a minion automorphism.

    For a minion $\minion{M}$, a \emph{minion core} of $\minion{M}$ is a minion core $\minion{N}$ such that $\minion{M} \sim \minion{N}$. 
\end{definition}

The following simple result shows that every ``finite'' minion has an essentially unique minion core, therefore we will talk about \emph{the} minion core of $\minion{M}$.

\begin{theorem}
    Every minion $\minion{M}$ with all the sets $\minion{M}^{(n)}$ finite has a minion core, which is unique up to minion isomorphisms. 
\end{theorem}

\begin{proof} 
  We start the proof of existence by inductively selecting minion homomorphisms $\xi_1$, $\xi_2$, \dots from $\minion{M}$ to $\minion{M}$  so that for every $i \in \mathbb{N}$, 
  \begin{itemize}
      \item $\xi_i^{(j)}(\minion{M}^{(j)}) \subseteq \xi_{i-1}^{(j)}(\minion{M}^{(j)})$ for every $j\in \mathbb{N}$ such that  $j \leq i-1$, and
      \item $(\xi_i^{(1)}(\minion{M}^{(1)}), \xi_i^{(2)}(\minion{M}^{(2)}), \dots, \xi_i^{(i)}(\minion{M}^{(i)}))$ is minimal with respect to component-wise inclusion.
  \end{itemize}
Since there exists a homomorphism from $\minion{M} \to \minion{M}$ (the identity), such a sequence exists.

Note that $\xi_i^{(j)}(\minion{M}^{(j)}) = \xi_{i'}^{(j)}(\minion{M}^{(j)})$ whenever $j \leq i' \leq i$ by minimality. We thin out  the sequence so that $\xi_1^{(1)} = \xi_2^{(1)} = \dots$, which is possible since there are only finitely many functions $\minion{M}^{(1)} \to \minion{M}^{(1)}$. Also note that thinning out preserves the two properties above. Next we thin out the sequence again so that $\xi_2^{(2)} = \xi_3^{(2)} = \dots$, and so on, until we achieve $\xi_i^{(j)} = \xi_{i'}^{(j)}$ whenever $j \leq i' \leq i$. 

We define $\xi = (\xi^{(i)})_i$ by $\xi^{(i)} = \xi_i^{(i)}$. It preserves minors as for any $m,n$, the $m$th and $n$th components of $\xi$ coincide with the corresponding components of $\xi_{\max \{m , n \}}$. We define $\minion{N}$ as the image of $\minion{M}$ under $\xi$, that is, $\minion{N}^{(n)} = \xi(\minion{M}^{(n)})$ and $\minion{N}^{(\alpha)}$ is a restriction of $\minion{M}^{(\alpha)}$  (minor preservation ensures that this makes sense). We have $\minion{M} \to \minion{N}$ as witnessed by $\xi$ and $\minion{N} \leq \minion{M}$ witnessed by the inclusion. Finally, observe that $\minion{N}$ is a minion core. Indeed, if some homomorphism $\nu: \minion{N} \to \minion{N}$ was not an automorphism, then $\nu^{(i)}$ would not be one-to-one for some $i$. But then $(\xi\nu\xi)^{(i)}(\minion{M}^{(i)}) = (\xi_i\nu\xi_i)^{(i)}(\minion{M}^{(i)})$ would be a proper subset of $\xi_i^{(i)}(\minion{M}^{(i)})$, contradicting the minimality.

 In order to show uniqueness, consider two minion cores $\minion{N}$ and $\minion{N}'$ of $\minion{M}$. By composing homomorphisms witnessing $\minion{N} \sim \minion{M} \sim \minion{N}'$ we get homomorphisms $\xi: \minion{N} \to \minion{N}'$ and $\nu: \minion{N}' \to \minion{N}$. We show that $\xi$ is invertible. Since $\minion{N}$ is a minion core, its endomorphism $\nu\xi$ is invertible and similarly $\xi\nu: \minion{N}' \to \minion{N}'$ is invertible as well. Therefore $\xi_1 = (\nu\xi)^{-1}\nu$ is a left inverse of $\xi$ and $\xi_2 = \nu(\xi\nu)^{-1}$ is a right inverse of $\xi$. They are necessarily equal as $\xi_1 = \xi_1\xi\xi_2 = \xi_2$, and we are done. 
\end{proof}

The significance of minion cores is thus in that they provide a canonical representative of each $\sim$-equivalence class (for ``finite'' minions). 

Even if one is primarily interested in comparing function clones (rather than minions) using minion homomorphisms, minions naturally appear since the minion core of a function clone constructed in the above proof is a function minion, but not necessarily a function clone. One such an example is discussed in \Cref{subsec:collapse}. 

The following consequence of \Cref{lem:binaries} will be useful for proving that a Boolean multisorted minion is a core, since it implies that we only need to care about binary parts of minion endomorphisms. 

\begin{lemma} \label{lem:cores-binaries}
Let $\minion M$ be a function minion on a pair of multisorted sets $(\tuple{A},\tuple{B})$. Let $m \in \mathbb{N}$ be greater than or equal to the size of each sort $A_i$. If for every minion homomorphism $\xi: \minion{M} \to \minion{M}$ its $m$-ary part $\xi^{(m)}: \minion{M}^{(m)} \to \minion{M}^{(m)}$ is a bijection, then $\minion{M}$ is a minion core.
\end{lemma}

\begin{proof}
   Consider a minion homomorphism $\xi: \minion{M} \to \minion{M}$. 
   By the assumption, the mapping $\xi^{(m)}$ is a bijection, therefore there exists $n$ such that 
    $$
    (\xi^{(m)})^n = \underbrace{\xi^{(m)} \circ \xi^{(m)} \circ \dots \circ \xi^{(m)}}_{\text{$n$ times}} = \mathrm{id}_{\minion M^{(m)}}
    $$
    (one can take e.g. $n=|\minion M^{(m)}|!$). It then follows from \Cref{lem:binaries} that  $\mu = \xi^{n-1}$  is a both-sided inverse to $\xi$. Indeed, we have $(\mu \circ \xi)^{(m)} = \mu^{(m)} \circ \xi^{(m)} = (\xi^{(m)})^{n-1} \circ \xi^{(m)} =  (\xi^{(m)})^n = \mathrm{id}_{{\minion M}^{(m)}}$, so $\mu \circ \xi = \mathrm{id}_{\minion M}$, and similarly $(\xi \circ \mu)^{(m)}  = \mathrm{id}_{{\minion M}^{(m)}}$, so $\xi \circ \mu = \mathrm{id}_{\minion M}$.
\end{proof}

\section{Boolean multisorted clones} \label{sec:classif}

In this section we work on the $k$-sorted Boolean set 
$$
\tuple{A} = (A_1, A_2, \dots, A_k), \quad A_1 = A_2 = \dots = A_k = \{0,1\},
$$
where $k$ is a positive integer. The operations on $\tuple{A}$ are called \emph{$k$-sorted Boolean operations}, typically denoted by $\tuple{f}$, $\tuple{g}$, or $\tuple{h}$. They are $k$-tuples of Boolean operations $\{0,1\}^n \to \{0,1\}$. Recall that clones on $\tuple{A}$ are called $k$-sorted Boolean clones. 

The goal of this section is to describe the ordering by minion homomorphisms in the class of all multisorted Boolean clones of the form $\Pol(\Gamma)$, where $\Gamma$ is a set of at most binary relations. This will be done by computing all possible minion cores of these clones  and the ordering between them.

\Cref{thm:rel-cores} allows us to concentrate on idempotent clones. We use the notation $\idemp_k$ instead of $\idemp(\tuple{A})$, or simply $\idemp$ if $k$ is clear from the context.
$$
\idemp_k = \mbox{ all idempotent $k$-sorted Boolean operations}
$$

We start the project by looking at a concrete example.

\begin{example} \label{ex:description-informal}
    Let $k=4$ and $\Gamma = \{R_1, R_2, R_3\}$, where
    \begin{align*}
        R_1 &= \{(0,1),(1,0)\}  \text{ of type  $(1,2)$} \\
        R_2 &= \{(0,0),(0,1),(1,1)\} \text{ of type $(2,3)$} \\ 
        R_3 &= \{(0,0),(0,1),(1,0)\}  \text{ of type $(3,4)$}. 
    \end{align*}
    
    What quadruples $\tuple{h}=(h_1,h_2,h_3,h_4)$ are in $\Polid(\Gamma)$? 
    First, $\tuple{h}$ preserves $R_1$ if, and only if, $h_1^d = h_2$, where $h_1^d$ is the dual of $h_1$ (\Cref{def:dual}). 
    Second, $\tuple{h}$ preserves $R_2$, which is the inequality relation $\leq$, if, and only if, $h_2(\tuple{a}) \leq h_3(\tuple{b})$ whenever $\tuple{a} \leq \tuple{b}$ (component-wise); this relation is denoted $h_2 \triangleleft h_3$  (\Cref{def:triangle}).
    Third, $\tuple{h}$ preserves $R_3$ if, and only if, $h_3 \triangleleft h_4^d$.
    Altogether, 
    $$
    \Polid(\Gamma) = \{\tuple{h} \in \idemp_4 \mid h_1^d = h_2 \triangleleft h_3 \triangleleft h_4^d\}.
    $$
\end{example}

The notation for Boolean operations is introduced in \Cref{subsec:boolean} together with some properties. The description similar to the example is formally defined and then simplified in \Cref{subsec:description}. 
In \Cref{subsec:collapse}, we introduce the minions appearing in \Cref{fig:order0} and prove that every multisorted Boolean clone that we consider is equivalent to one of them. \Cref{subsec:homo} shows that all of these minions are minion cores and establishes some (non-)inequalities between them. \Cref{subsec:summary} then finishes the job.

\subsection{Boolean operations} \label{subsec:boolean}

We denote by $\leq$ the natural ordering of $\{0,1\}$, i.e., $0 \leq 1$. This ordering is extended to tuples and Boolean operations: for $\tuple{a}, \tuple{b} \in \{0,1\}^n$, and $f,g: \{0,1\}^n \to \{0,1\}$ we define
\begin{itemize}
    \item $\tuple{a} \leq \tuple{b}$ if $a_i \leq b_i$ for every $i \in [n]$, and
    \item $f \leq g$ if $f(\tuple{a}) \leq g(\tuple{a})$ for every $\tuple{a} \in \{0,1\}^n$.
\end{itemize}
We denote by $\wedge$ and $\vee$ the binary minimum and maximum operation on $\{0,1\}$, respectively, and extend it to  Boolean operations. 
\begin{itemize}
    \item $a \wedge b = \min \{a,b\}$, $a \vee b = \max \{a,b\}$
    \item $(f \wedge g)(\tuple{a}) = f(\tuple{a}) \wedge g(\tuple{a})$, $(f \vee g)(\tuple{a}) = f(\tuple{a}) \vee g(\tuple{a})$ for every $\tuple{a} \in \{0,1\}^n$
\end{itemize}
Notice that $\vee$ and $\wedge$ are the only idempotent binary Boolean operations apart of the two projections.

We denote by $\overline{a}$ the ``negation'' of $a$ and extend the notation to tuples:
\begin{itemize}
    \item $\overline{a} = 1-a$,
    \item $\overline{(a_1, a_2, \dots, a_n)} = (\overline{a_1},\overline{a_2}, \dots, \overline{a_n})$.
\end{itemize}
Finally, we introduce the relations between Boolean operations appearing in \Cref{ex:description-informal}.

\begin{definition} \label{def:triangle}
Let $f$ and $g$ be two $n$-ary Boolean operations. We write $f \triangleleft g$ if $f(\tuple a) \leq g(\tuple b)$ for all $n$-tuples $\tuple a, \tuple b \in \{0,1\}^n$ such that $\tuple a \leq \tuple b$.
\end{definition}

\noindent
Note that $f \triangleleft g$ implies $f \leq g$ and that the converse is not true in general. For instance, while $f \leq f$ holds for every Boolean operation, the relation $f \triangleleft f$ is nontrivial -- it is equivalent to $f$ being monotone. 
An alternative definition of $f \triangleleft g$ is that there exists a monotone $g'$ such that $f \leq g' \leq g$.

\begin{definition} \label{def:dual}
Let $f$ be an $n$-ary Boolean operation. The \emph{dual} of $f$, denoted by $f^d$, is defined as follows.
$$
f^d (\tuple a) = \overline{f(\overline{\tuple a})} \quad \text{for every $\tuple a \in \{0,1\}^n$}
$$
\end{definition}

Notice that $(f^d)^d=f$, that $(f \vee g)^d = f^d \wedge g^d$, and that $(f \wedge g)^d = f^d \vee g^d$. 
We now list several other properties. They are all easy to verify and we omit the proof.
\begin{description}
    \item[Skew symmetry] $f \triangleleft g^d$  iff $g \triangleleft f^d$.
    \item[Strong transitivity] If $f \leq g \triangleleft h$ or $f \triangleleft g \leq h$, then $f \triangleleft h$.
    \item[Compatibility] If $f \triangleleft g$ and $f' \triangleleft g'$, then $f \wedge f' \triangleleft g \wedge g'$ and $f \vee f' \triangleleft g \vee g'$.
\end{description}

The following lemma will enable us to
project a self-dual operation $g$ to an interval $f \leq f^d$.

\begin{lemma}\label{lem:h}
Let $f$ and $g$ be operations such that $g = g^d$ and $
f \leq f^d$. Define $h = (g \vee f) \wedge f^d$. Then $f \leq h = h^d$.
\end{lemma}

\begin{proof}
We have $h^d = ((g \vee f) \wedge f^d)^d = (g \vee f)^d \vee f = (g^d \wedge f^d) \vee f = (g^d \vee f) \wedge (f^d \vee f) = (g \vee f) \wedge f^d = h$. 
Also $h = h^d = (g \wedge f^d) \vee f \geq f$. 
\end{proof}

\subsection{Description} \label{subsec:description}

In this subsection we show that the clones of our interest can be described by means of $\triangleleft$, $=$, and duals (\Cref{lem:description_exists}) and then we simplify the description in \Cref{thm:reduced-description}. 

\begin{definition}[Description]
A description over a sequence of symbols $(\dd{h}_1, \dd{h}_2,$ $ \dots, \dd{h}_k)$, $k \in \mathbb{N}$ is a set of formal expressions, called \emph{constraints}, that can take one of the following forms
$$
  \dd{h}_i \triangleleft \dd{h}_j, \ 
  \dd{h}_i \triangleleft \dd{h}_j^d, \ 
  \dd{h}_i^d \triangleleft \dd{h}_j, \ 
  \dd{h}_i^d \triangleleft \dd{h}_j^d, \ 
  \dd{h}_i = \dd{h}_j, \
  \dd{h}_i = \dd{h}_j^d, \
  \dd{h}_i^d = \dd{h}_j, \
  \dd{h}_i^d = \dd{h}_j^d,
$$
where $i,j \in [k]$ (not necessarily distinct). 

If $\descript$ is a description over $(\dd{h}_1, \dots, \dd{h}_k)$,
then we define $\Clo(\descript)$ as the set of all idempotent $k$-sorted Boolean operations $(h_1, \dots, h_k) \in \idemp_k$ which satisfy all the constraints in $\descript$ in the obvious sense.
We also say that $\descript$ \emph{describes} $\Clo(\descript)$. 
\end{definition}

\begin{example} 
  The clone from \Cref{ex:description-informal} is
  $$
    \Clo(\dd{h}_1^d = \dd{h}_2,\dd{h}_2 \triangleleft \dd{h}_3,\dd{h}_3 \triangleleft \dd{h}_4^d) = \{\tuple{h} \in \idemp_4 \mid h_1^d = h_2 \triangleleft h_3 \triangleleft h_4^d\}.
  $$  
  We will also use a shorter notation such as, e.g.,
$$
\Clo(\dd{h}_1^d = \dd{h}_2 \triangleleft \dd{h}_3 \triangleleft \dd{h}_4^d)
$$  
\end{example}

\begin{lemma} \label{lem:description_exists}
For each set $\Gamma$ of at most binary relations on $\tuple{A}$, the clone $\Pol(\Gamma)$ is equivalent to $\clone{T}$ or to $\Clo(\descript)$ for some description $\descript$.
\end{lemma}

\begin{proof}
By \Cref{thm:rel-cores}, it is enough to show that $\Polid(\Gamma)$ is equal to $\Clo(\descript)$ for some $\descript$. 
Indeed, unless $\Pol(\Gamma) \sim \clone{T}$, the theorem gives us a multisorted set $\tuple{B}$, which is $l$-sorted Boolean by the second item, and a set of at most binary relations $\Delta$ on $\tuple B$ such that $\Pol(\Gamma) \sim \Polid(\Delta)$. The proof (applied to $\Delta$ instead of $\Gamma$) will give us $\Polid(\Delta) = \Clo(\descript)$ for some $\descript$, so we obtain $\Pol(\Gamma) \sim \Clo(\descript)$.

Let $k$ be the number of sorts of $\tuple{A}$.
We construct $\descript$ over $(\dd{h}_1, \dd{h}_2,$ $ \dots, \dd{h}_k)$ by including the constraint  $\dd{h}_i \triangleleft \dd{h}_j^d$ for each relation $\{0,1\}^2\setminus\{(1,1)\}$ of type $(i,j)$ in $\Gamma$,
 $\dd{h}_i^d \triangleleft \dd{h}_j$ for $\{0,1\}^2\setminus\{(0,0)\}$, $\dd{h}_i \triangleleft \dd{h}_j$ for $\{0,1\}^2\setminus\{(1,0)\}$, $\dd{h}_i^d \triangleleft \dd{h}_j^d$ for $\{0,1\}^2\setminus\{(0,1)\}$, $\dd{h}_i = \dd{h}_j$ for $\{ (0,0), (1,1) \}$, and $\dd{h}_i = \dd{h}_j^d$ for $\{ (0,1), (1,0) \}$, ignoring any relation not mentioned.

The inclusion $\Polid(\Gamma) \subseteq \Clo(\descript)$ follows from the definitions. For the other direction, it is enough to additionally observe that we only ignored unary relations $\emptyset$, $\{a\}$ and binary relations $\emptyset$, $\{0,1\}^2$, $\{(a,b),(a,c)\}$, $\{(b,a),(c,a)\}$ ($a,b,c \in \{0,1\}$), which are all preserved by any idempotent operation.
\end{proof}

Now we move on to the task of making the description simpler. In particular, we eliminate expressions of the form $\dd{h}_i^d \triangleleft \dd{h}_j$ from $\descript$ and  make $=$ appear only in a restricted way.

\begin{definition}[Reduced form]
A description in \emph{reduced form} is a description over $(\dd{f}_1, \dots, \dd{f}_n, \dd{g}_1, \dots, \dd{g}_m)$, where $m$ and $n$ are nonnegative integers with $m+n \geq 1$, such that
\begin{enumerate}[label=(\roman*)]
\item \label{itm:i}
$\dd{g}_i = \dd{g}_i^d$ is in $\descript$  for each $i \in [m]$,
\item all the remaining constraints in $\descript$ are of the form $\dd{f}_i \triangleleft \dd{f}_j$, $\dd{f}_i \triangleleft \dd{f}_j^d$, $\dd{f}_i \triangleleft \dd{g}_j$, or
$\dd{g}_i \triangleleft \dd{g}_i$, and
\item \label{itm:ii} there are no $\triangleleft$-cycles of length more than 1, that is, there are no chains $\dd{f}_{i_1} \triangleleft \dd{f}_{i_2} \triangleleft \dots \triangleleft \dd{f}_{i_l} \triangleleft \dd{f}_{i_1}$ in $\descript$ with $l > 1$ and $i_1, \dots, i_l$ pairwise distinct. (Here, again, the membership in $\descript$ means that $\dd{f}_{i_1} \triangleleft \dd{f}_{i_2}$ is in $\descript$, $\dd{f}_{i_2} \triangleleft \dd{f}_{i_3}$  is in $\descript$, etc.)
\end{enumerate}
\end{definition}

We aim to show that each $\Clo(\descript)$ is equivalent to $\Clo(\descript')$, where the description $\descript'$ is reduced.

Minion homomorphisms will be defined by formulas such as
$$
\xi(h_1, h_2, h_3) = (h_3, h_2^d, h_1, h_1)
$$
and later by more complex formulas, e.g.,
$$
\xi(h_1, h_2, h_3) = (h_2^d, (h_3 \wedge h_2) \vee h_3^d).
$$
We show that each such a formula defines a minor preserving map.

\begin{lemma} \label{lem:preserving-minors}
   Let $k,l \in \mathbb{N}$. 
    Let $\dd{t}_1$, $\dd{t}_2$, \dots, $\dd{t}_{l}$ be terms over the set of symbols $\{\dd{h}_1, \dots, \dd{h}_k\}$ in the signature $\{\wedge,\vee,\ ^{d}\}$. For a $k$-sorted Boolean $n$-ary operation $(h_1, \dots, h_k)$ and $j \in [l]$ we define $\dd{t}_j(h_1, \dots, h_k): \{0,1\}^n \to \{0,1\}$ in the natural way (replace $\dd{h}_i$ by $h_i$ and compute the expression). Then
    the collection $\xi = (\xi^{(n)}: \idemp_k^{(n)} \to \idemp_{l}^{(n)})_{n}$ defined by
    $$
    \xi(h_1, \dots, h_k) = (\dd{t}_1(h_1, \dots, h_k), \dd{t}_2(h_1, \dots, h_k), \dots,
    \dd{t}_{l}(h_1, \dots, h_k))
    $$
    is a minion homomorphism $\idemp_k \to \idemp_l$.
\end{lemma}

\begin{proof} 
    First we claim that for every term $\dd{s}$ over $\{\dd{h}_1,\dots,\dd{h}_k\}$ in the signature $\{\wedge,\vee,\ ^{d}\}$ there exists a Boolean operation $\Tilde{\dd{s}}$ of arity $2k$ such that for every $n \in \mathbb{N}$, every $k$-tuple of $n$-ary Boolean operations $\tuple{h}=(h_1, \dots, h_k)$, and every $\tuple{a} \in \{0,1\}^n$, we have
$$
    (\dd{s}(\tuple{h}))(\tuple{a}) = \Tilde{\dd{s}}(h_1(\tuple{a}), \dots, h_k(\tuple{a}), h_1(\overline{\tuple{a}}), \dots, h_k(\overline{\tuple{a}}))).
$$
    The claim is proved by induction of the depth of $\dd{s}$. The base case when  $\dd{s} = \dd{h}_j$ for some $j$ is clear. If $\dd{s} = \dd{s}_1 \wedge \dd{s}_2$, we have
\begin{align*}
    (\dd{s}(\tuple{h}))(\tuple{a}) 
    &= (\dd{s}_1(\tuple{h}) \wedge \dd{s}_2(\tuple{h}))(\tuple{a}) 
    = (\dd{s}_1(\tuple{h}))(\tuple{a}) \wedge (\dd{s}_2(\tuple{h}))(\tuple{a}) \\
    &= \Tilde{\dd{s}}_1(h_1(\tuple{a}), \dots, h_k(\tuple{a}),h_1(\overline{\tuple{a}}), \dots, h_k(\overline{\tuple{a}})) \\ &\quad \wedge \Tilde{\dd{s}}_2(h_1(\tuple{a}), \dots, h_k(\tuple{a}),h_1(\overline{\tuple{a}}), \dots, h_k(\overline{\tuple{a}})),
\end{align*}
so we can define
$$
\Tilde{\dd{s}}(a_1, \dots, a_k, b_1, \dots, b_k) = \tilde{\dd{s}}_1(a_1, \dots, a_k, b_1, \dots, b_k) \wedge \tilde{\dd{s}}_2 (a_1, \dots, a_k, b_1, \dots, b_k).
$$
The proof for $\dd{s} = \dd{s}_1 \vee  \dd{s}_2$ is completely analogous. Finally, if $\dd{s} = \dd{s}_1^d$, then we have
\begin{align*}
    (\dd{s}(\tuple{h}))(\tuple{a}) 
    &= (\dd{s}_1^d(\tuple{h}))(\tuple{a}) 
    = \overline{(\dd{s}_1(\tuple{h}))(\overline{\tuple{a}})}\\
    &= \overline{ \Tilde{\dd{s}}_1(h_1(\overline{\tuple{a}}), \dots, h_k(\overline{\tuple{a}}), h_1(\tuple{a}), \dots, h_k(\tuple{a}))}
\end{align*}
so we can define
$$
\Tilde{\dd{s}}(a_1, \dots, a_k, b_1, \dots, b_k) = \overline{\Tilde{\dd{s}}_1(b_1, \dots, b_k, a_1, \dots, a_k)}.
$$
    
    The mapping $\xi$ clearly  preserves arities. It also easy to see  that $\dd{t}_j(h_1, \dots, h_k)$ is idempotent, so $\xi$ is correctly defined. It remains to verify that that for every $n\in\mathbb{N}$, every $n$-ary $k$-sorted $\tuple{h}=(h_1, \dots, h_k) \in \idemp_k$, every $m \in \mathbb{N}$, and every $i_1$, \dots, $i_n \in [m]$, we have
    $$
    \xi(\tuple{h}) \circ (\tuple{\pi}^m_{i_1}, \tuple{\pi}^m_{i_2}, \dots, \tuple{\pi}^m_{i_n}) = \xi(\tuple{h} \circ (\tuple{\pi}^m_{i_1}, \tuple{\pi}^m_{i_2}, \dots, \tuple{\pi}^m_{i_n})).
    $$
    Both sides are $l$-tuples of $m$-ary Boolean operations. We need to show that for each $j \in [l]$, the $j$th Boolean operations are the same on both sides. That is, we need to prove
    $$
    \dd{t}_j(\tuple{h}) \circ (\pi^m_{i_1}, \pi^m_{i_2}, \dots, \pi^m_{i_n}) = \dd{t}_j(\tuple{h} \circ (\tuple{\pi}^m_{i_1}, \tuple{\pi}^m_{i_2}, \dots, \tuple{\pi}^m_{i_n})).
    $$
    For any tuple $\tuple{a} = (a_1, \dots, a_m) \in \{0,1\}^m$ we have 
    \begin{align*}
    (\dd{t}_j(\tuple{h}) \circ &(\pi^m_{i_1}, \dots, \pi^m_{i_n}))(\tuple{a}) \\
      &= \dd{t}_j(\tuple{h})(\pi^m_{i_1}(\tuple{a}), \dots \pi^m_{i_n}(\tuple{a})) \\
      &= \dd{t}_j(\tuple{h})(a_{i_1}, \dots, a_{i_n}) \\
      &= \tilde{\dd{t}}_j(h_1(a_{i_1}, \dots, a_{i_n}), \dots, h_k(a_{i_1}, \dots, a_{i_n}), \\ 
      &\quad\quad\quad h_1(\overline{a_{i_1}}, \dots, \overline{a_{i_n}}), \dots, h_k(\overline{a_{i_1}}, \dots, \overline{a_{i_n}}))\\
      &=  \tilde{\dd{t}}_j((h_1 \circ (\pi^m_{i_1},\dots,\pi^m_{i_n}))(\tuple{a}), \dots, (h_k \circ (\pi^m_{i_1},\dots,\pi^m_{i_n}))(\tuple{a}), \\
         & \quad \quad \quad (h_1 \circ (\pi^m_{i_1},\dots,\pi^m_{i_n}))(\overline{\tuple{a}}), \dots, (h_k \circ (\pi^m_{i_1},\dots,\pi^m_{i_n}))(\overline{\tuple{a}}))\\
      &=  (\dd{t}_j(\tuple{h} \circ (\tuple{\pi}^m_{i_1},  \dots, \tuple{\pi}^m_{i_n})))(\tuple{a}),
    \end{align*}
    as required.
\end{proof}

The second ingredient is a criterion for satisfiability of a 2-CNF from \cite{ASPVALL1979121}, which we now state. \emph{2-CNF} is a formula of propositional logic of a specific form: it is a conjunction of clauses, each clause is a disjunction of two literals, and each literal is either a variable or a negated variable. Such a formula is e.g.
$$
(x \OR y) \AND (\neg y \OR z) \AND (y \OR y).
$$
(We use a nonstandard notation for disjunction and conjunction so that we do not overload the symbols $\wedge$ and $\vee$ too much.)
A 2-CNF is \emph{satisfiable} if there exists an assignment $\phi: \mathrm{variables} \to \{\true,\false\}$ making the formula true.

The \emph{implication graph} of a 2-CNF is the following directed graph. Vertices are the variables and their negations. For each clause, the graph contains the edges corresponding to the implications (typically two) that are logically equivalent to the clause; and there are no other edges. In the example above, the implication graph has edges 
$$
\neg x \to y, \ 
\neg y \to x, \ 
y \to z, \ 
\neg z \to \neg y, \ 
\neg y \to y.
$$

The criterion for satisfiability is as follows.

\begin{theorem}[\cite{ASPVALL1979121}] \label{thm:cnf}
A 2-CNF formula is satisfiable if and only if there is no variable $x$ such that there exists a directed walk $x \to \dots \to \neg x$ and a directed walk $\neg x \to \dots \to x$. 
\end{theorem}

We are ready to reduce the descriptions. 

\begin{theorem} \label{thm:reduced-description}
    For every description  $\descript$, there exists a description $\descript'$ in a reduced form such that
    $\Clo(\descript) \sim \Clo(\descript')$.
\end{theorem}

\begin{proof}
   Let $\descript$ be a description over $\{\dd{h}_1, \dots, \dd{h}_k\}$.
   Without loss of generality, assume that there is no description over a smaller set of symbols which describes a multisorted clone equivalent to $\Clo(\descript)$ (if there is such, we can replace the original description with one which has the smallest number of symbols and describes an equivalent clone). 

   A consequence of this \emph{minimality assumption} is that there are no $i,j \in [k]$ with $i \neq j$ such that every $(h_1, \dots, h_k) \in \Clo(\descript)$ satisfies $h_i = h_j$. Indeed, if there are such $i$ and $j$, say $i=1$ and $j=2$, then we define $\descript'$ over the smaller set of symbols $\{\dd{h}_2, \dots, \dd{h}_k\}$  by replacing every occurrence of $\dd{h}_1$ by $\dd{h}_2$ and every occurrence of $\dd{h}_1^d$ by $\dd{h}_2^d$. The mapping $\xi: \Clo(\descript) \to \Clo(\descript')$  defined by 
$\xi(h_1, \dots, h_k) =(h_2, \dots h_k)$ is correctly defined: the constraints in $\descript'$ are satisfied since $h_1$ is always equal to $h_2$ when $(h_1, \dots, h_k) \in \Clo(\descript)$. Also $\xi$ preserves minors by \Cref{lem:preserving-minors}, therefore it is a minion homomorphism. On the other hand, it is also easy to see that $\xi: \Clo(\descript') \to \Clo(\descript)$ defined by $\xi(h_2, h_3, \dots, h_k)=(h_2, h_2, h_3, \dots, h_k)$ is a minion homomorphism. Therefore $\Clo(\descript)$ is equivalent to $\Clo(\descript')$, a contradiction to the minimality assumption.

Similarly, there are no $i,j \in [k]$ with $i \neq j$ such that every $(h_1, \dots, h_k) \in \Clo(\descript)$ satisfies $h_i = h_j^d$. The differences in the above argument is that while creating $\descript'$ we replace $\dd{h}_1$ by $\dd{h}_2^d$ (instead of $\dd{h}_2$) and $\dd{h}_1^d$ by $\dd{h}_2$, and we define the second homomorphism by $\xi(h_2, h_3, \dots, h_k)=(h_2^d,h_2, h_3, \dots, h_k)$. The  homomorphisms are correctly defined since $(h_1^d)^d = h_1$.

Now we change the set of symbols $\{\dd{h}_1$, \dots, $\dd{h}_k\}$ of $\descript$ (and change the constraints accordingly) to $\{\dd{f}_1, \dots, \dd{f}_n, \dd{g}_1, \dots, \dd{g}_m\}$ ($n+m=k$) so that for each $(f_1, \dots, f_n,g_1, \dots, g_m) \in \Clo(\descript)$ we have $g_1=g_1^d$, \dots, $g_m = g_m^d$ and $m$ is the largest with this property, i.e., for any $i \in [n]$ we have $f_i \neq f_i^d$ for some $(f_1, \dots, f_n,g_1, \dots, g_m) \in \Clo(\descript)$. 

We make several adjustments to $\descript$, none of which changes $\Clo(\descript)$ because of trivial reasons.
We add to $\descript$ all the constraints $\dd{g}_i=\dd{g}_i^d$ ($i \in [m]$),
remove redundant constraints of the form $\dd{f}_i=\dd{f}_i$, $\dd{f}_i^d=\dd{f}_i^d$, $\dd{g}_i = \dd{f}_i$, $\dd{g}_i^{(d)}=\dd{g}_i$, and $\dd{g}_i^d=\dd{g}_i^d$, 
 replace constraints $\dd{f}_i^d \triangleleft \dd{f}_j^d$ by $\dd{f}_j \triangleleft \dd{f}_i$ (which is justified by skew symmetry),
 replace $\dd{g}_i^d$ by $\dd{g}_i$ in every $\triangleleft$-constraint,
 and replace $\dd{g}_i \triangleleft \dd{f}_j$ by $\dd{f}_j^d \triangleleft \dd{g}_i$ and $\dd{g}_i \triangleleft \dd{f}_j^d$ by $\dd{f}_j \triangleleft \dd{g}_i$.
 We refer to these adjustments of $\descript$ as \emph{simple adjustments}.

The only $=$-constraints left are $\dd{g}_i = \dd{g}_i^d$ since $\dd{f}_i = \dd{f}_i^d$ is not in $\descript$ by the choice of $m$ and all the other $=$-constraints were either removed or are impossible by the minimality assumption. Note also that $\dd{g}_i \triangleleft \dd{g}_j$ with $i \neq j$  is not in $\descript$, otherwise every $(f_1, \dots, f_n,g_1, \dots, g_m) \in \Clo(\descript)$ satisfies $g_i \triangleleft g_j = g_j^d \triangleleft g_i^d = g_i$, therefore $g_i=g_j$ by antisymmetry,
 a contradiction to the minimality assumption. The remaining $\triangleleft$-constraints of $\descript$ are thus of the form $\dd{f}_i \triangleleft \dd{f}_j$, $\dd{f}_i \triangleleft \dd{f}_j^d$, $\dd{f}_i^d \triangleleft \dd{f}_j$, $\dd{f}_i \triangleleft \dd{g}_j$, $\dd{f}_i^d \triangleleft \dd{g}_j$, or $\dd{g}_i \triangleleft \dd{g}_i$. We need to get rid of the  cases $\dd{f}_i^d \triangleleft \dd{f}_j$ and $\dd{f}_i^d \triangleleft \dd{g}_j$. 

The description $\descript'$ will be obtained from $\descript$ by selecting some  numbers $j \in \mathbb{N}$ and replacing $\dd{f}_j$ by $\dd{f}_j^d$ and vice versa. Note that whatever set of numbers we take, the obtained clone $\Clo(\descript')$ is equivalent to $\Clo(\descript)$ via the minion homomorphisms defined in both directions as $\xi(f_1, \dots, f_k,g_1, \dots, g_k) = \xi(f_1', \dots, f_k',g_1, \dots, g_k)$, where $f_j'=f_j^d$ if $j$ was selected and $f_j'=f_j$ otherwise. The numbers will be selected according to a satisfying assignment of a 2-CNF defined as follows. The set of variables of the 2-CNF is $\{\dd{f}_1, \dots, \dd{f}_n\}$ and
\begin{itemize}
    \item for each constraint $\dd{f}_i \triangleleft \dd{f}_j$ in $\descript$ we add to the 2-CNF the clause $(\neg \dd{f}_i \OR \dd{f}_j)$,

    \item for each $\dd{f}_i \triangleleft \dd{f}_j^d$ in $\descript$ we add $(\neg \dd{f}_i \OR \neg \dd{f}_j)$,

    \item for each $\dd{f}_i^d \triangleleft \dd{f}_j$ in $\descript$ we add $( \dd{f}_i \OR \dd{f}_j)$,

    \item for each $\dd{f}_i \triangleleft \dd{g}_j$  in $\descript$ we add $(\neg \dd{f}_i \OR \neg \dd{f}_i)$, and

    \item for each $\dd{f}_i^d \triangleleft \dd{g}_j$ in $\descript$ we add $(\dd{f}_i \OR \dd{f}_i)$.
\end{itemize}

We claim that if $\dd{f}_i \to \dd{f}_j$ ($\dd{f}_i \to \dd{f}_j^d$, $\dd{f}_i^d \to \dd{f}_j$, $\dd{f}_i^d \to \dd{f}_j^d$) in the implication graph of the 2-CNF formula, then the corresponding relation $f_i \triangleleft f_j$ ($f_i \triangleleft f_j^d$, $f_i^d \triangleleft f_j$,  $f_i^d \triangleleft f_j^d$) holds for every $(f_1, \dots, f_n, g_1, \dots, g_m) \in \Clo(\descript)$.
This follows from properties of $\triangleleft$ and duals 
as follows. In the first item (i.e., $\dd{f}_i \triangleleft \dd{f}_j$ in $\descript$) we added $(\neg \dd{f}_i \OR \dd{f}_j)$ to the 2-CNF which gives $\dd{f}_i \to \dd{f}_j$ and $\dd{f}_j^d \to \dd{f}_i^d$ in the implication graph; we have $f_i \triangleleft f_j$ (since $\dd{f}_i \triangleleft \dd{f}_j$ in $\descript$) and $f_j^d \triangleleft f_i^d$ (by skew symmetry)
for every $(f_1, \dots, g_m) \in \Clo(\descript)$. The second and third items are similar. The fourth item ($\dd{f}_i \triangleleft \dd{g}_j$) gives $\dd{f}_i \to (\neg \dd{f}_i)$ in the implication graph and for any $(f_1, \dots, g_m) \in \Clo(\descript)$ we have $f_i \triangleleft g_j=g_j^d \triangleleft f_i^d$. The fifth item is similar. 

It follows that there is no symbol $\dd{f}_i$ such that $\dd{f}_i \to \dots \to \neg \dd{f}_i$ and $\neg\dd{f}_i \to \dots \to \dd{f}_i$ in the implication graph. Indeed, otherwise $f_i \triangleleft \dots \triangleleft f_i^d \triangleleft \dots \triangleleft f_i$ for every $(f_1, \dots, g_m) \in \Clo(\descript)$, so $f_i = f_i^d$, a contradiction to the choice of $m$. 
Let $\phi$ be a satisfying assignment to our 2-CNF guaranteed by \Cref{thm:cnf}.

We create $\descript'$ by replacing $\dd{f}_j$ by $\dd{f}_j^d$ and $\dd{f}_j^d$ by $\dd{f}_j$ for every $j$ such that $\phi(\dd{f}_j)=\true$, and making the simple adjustments. Observe that the only $=$-constraints in $\descript'$ are still $\dd{g}_i = \dd{g}_i^d$ (and we have all such in $\descript'$) and that all $\triangleleft$-constraints in $\descript'$ are of the form $\dd{f}_i \triangleleft \dd{f}_j$, $\dd{f}_i \triangleleft \dd{f}_j^d$, $\dd{f}_i \triangleleft \dd{g}_j$, or 
$\dd{g}_i \triangleleft \dd{g}_i$ by the choice of the 2-CNF and the adjustments. It remains to observe that there are no cycles $\dd{f}_{i_1} \triangleleft \dd{f}_{i_2} \triangleleft \dots \triangleleft \dd{f}_{i_l} \triangleleft \dd{f}_{i_1}$ in $\descript'$ with $l > 1$ since otherwise $f_{i_1} \triangleleft f_{i_2} \triangleleft \dots \triangleleft f_{i_l} \triangleleft f_{i_1}$ and thus $f_{i_1}=f_{i_2}=\dots=f_{i_l}$ for every $(f_1, \dots, g_m) \in \Clo(\descript')$, which is impossible by the minimality assumption. Now the description $\descript'$ is reduced and $\Clo(\descript')$ is equivalent to $\Clo(\descript)$. The proof is concluded.
\end{proof}

\subsection{Collapse} \label{subsec:collapse}

In this subsection we show that every multisorted Boolean clone of the form $\Clo(\descript)$, where $\descript$ is in reduced form, is equivalent to a multisorted minion from an explicit collection that we now introduce.

\begin{definition} \label{def:thecores}
For a positive integer $k$ we define multisorted Boolean minions as follows.
    \begin{align*}
        \minAk &= \{(h_1, h_2, \dots, h_k) \in \idemp_k \mid h_1 \triangleleft h_2  \triangleleft \dots \triangleleft h_k \leq h_k^d\} \\ 
        \minBk &= \{(h_1, h_2, \dots, h_k) \in \idemp_k \mid h_1 \triangleleft h_2  \triangleleft \dots \triangleleft h_k \triangleleft h_k^d\} \\
        \minCk &= \{(h_1, h_2, \dots, h_k) \in \idemp_k \mid h_1 \triangleleft h_2 \triangleleft \dots \triangleleft h_{k-1} \leq h_{k} = h_{k}^d, \ h_{k-1} \triangleleft h_{k-1}^d\}  \\        
        \minDk &= \{(h_1, h_2, \dots, h_k) \in \idemp_k \mid h_1 \triangleleft h_2  \triangleleft \dots \triangleleft h_k = h_k^d\} \\
        \minBinf &= \{(h) \in \idemp_1 \mid h \triangleleft h \triangleleft h^d \} \\
        \minCinf &= \{ (h_1, h_2) \in \idemp_2 \mid h_1 \triangleleft h_1 \triangleleft h_2 = h_2^d\} \\ 
        \minDinf &= \{(h) \in \idemp_1 \mid h \triangleleft h = h^d\}
\end{align*}
\end{definition}

Definition for $\minCx{1}$ might be unclear: it should be read as $\minCx{1}=\minDx{1} = \{(h) \in \idemp_1 \mid h=h^d\}$. This minion will thus not appear in the classification.  One may also define the missing $\minAinf$ as $\minAinf=\minBinf = \{(h) \in \idemp_1 \mid h \triangleleft h \leq h^d\}$. 

The notation $\minAk$, \dots for the minions should be viewed as temporary. It was chosen simply because these are initial letters of the alphabet, there is no deeper reason behind it. Note that $\minAk$ has nothing to do with the $k$-sorted set $\tuple{A}$ and that  $\minDk$ has nothing to do with description $\descript$. We hope that these almost-collisions will not cause confusion.

Not all of these multisorted function minions are multisorted function clones. This is caused by $\leq$ in the definitions. We remark that $f\leq g$ iff $(f,g)$ preserves the pair of relations $(\{(0,0),(1,1)\},\{(0,0),(0,1),(1,1)\})$ of type $(1,2)$.
The simplest and interesting example of such a minion is $\minAx{1} = \{(h) \in \idemp_1 \mid h \leq h^d\}$ determined by the pair of relations $(\{(0,1),(1,0)\},\{(0,0),(0,1),(1,0)\})$. In other words, it consists of (1-tuples of) Boolean operations that avoid $\vee$ as a minor. It is the minion core of $\idemp_k$ for any $k$ and it is the largest element of the whole minion homomorphism poset that is strictly below $\minion{T}$. The minion was also identified in \cite{weaklyterminal} (see Theorem 3.8 there and remarks below). 

Also notice that all the definitions can be equivalently presented in a more symmetric way. For instance, $\tuple{h} \in \minAk$ if, and only if, 
$h_1 \triangleleft h_2  \triangleleft \dots \triangleleft h_k \leq h_k^d \triangleleft \dots \triangleleft h_2^d \triangleleft h_1^d.$
Observe further that every $\tuple{h}$ from any $k$-sorted minion from the list satisfies
$
h_1 \triangleleft h_2 \triangleleft \dots \triangleleft h_{k-1} \leq h_k \leq h_k^d \leq h_{k-1}^d \triangleleft \dots \triangleleft h_2^d \triangleleft h_1^d
$
and $h_i \triangleleft h_i^d$ for every $i \in [k-1]$. 
It also follows that $h_i \leq h_i^d$ for every $i \in [k]$. 

For each of these specific minions we give a sufficient condition under which a multisorted Boolean clone described by a reduced description is equivalent to it. It will follow from the subsequent results that this condition is also necessary, because the above multisorted minions are pairwise inequivalent (where $\minCk$ is only considered for $k \geq 2$). 
We need few concepts to state the results.

\begin{definition} \label{def:rank}
Let $\descript$ be a description over $(\dd{f}_1, \dots, \dd{f}_n, \dd{g}_1, \dots, \dd{g}_m)$ in a reduced form. 

We say that a function symbol $\dd{f}_i$ is \emph{monotone} if $\dd{f}_i \triangleleft \dd{f}_i$ is in $\descript$ and similarly for $\dd{g}_i$.

We define the remaining concepts only for $\descript$ without monotone symbols.
A \emph{path of length $l$} to $\dd{f}_i$ is a sequence
$\dd{f}_{i_1} \triangleleft \dd{f}_{i_2} \triangleleft \dots \triangleleft \dd{f}_{i_{l-1}} \triangleleft \dd{f}_{i_l}$ in $\descript$ where $i=i_l$.
We define the \emph{rank} of $\dd{f}_i$, denoted as $\rank(\dd{f}_i)$, to be the length of the longest such path plus one, i.e.,
    \[
    \rank (\dd{f}_i) = \max \{l \in \mathbb{N} \mid \dd{f}_{i_1} \triangleleft \dd{f}_{i_2} \triangleleft \dots \triangleleft \dd{f}_{i_{l-1}} \triangleleft \dd{f}_{i_l}  \mbox{ is a path to $\dd{f}_i$}\}.
    \]
Finally, the \emph{chain rank} of $\descript$ is defined as
\begin{align*}
\chrd = \max ( &\{\rank(\dd{f}_i) \mid i \in [n]\} \ \cup \\
                 & \{ \rank(\dd{f}_i)+\rank(\dd{f}_j) \mid i,j \in [n], \ \dd{f}_i \triangleleft \dd{f}_j^d \text{ is in } \descript \} \ \cup \\
                 & \{ 2\rank(\dd{f}_i) + 1 \mid i \in [n], \ \dd{f}_i \triangleleft \dd{g}_j \text{ is in $\descript$ for some } j \in [m]\} )
\end{align*}
\end{definition}

The rank of a symbol $\dd{f}_i$ is well defined since the description is reduced and there are no monotone symbols, so there are no $\triangleleft$-cycles.
Also note that if $\dd{f}_i \triangleleft \dd{f}_j$ ($i \neq j$), then $\rank(\dd{f}_i) < \rank(\dd{f}_j)$ since we can append $\dd{f}_j$ to a longest path to $\dd{f}_i$ and get a longer path to $\dd{f}_j$.

\begin{theorem} \label{thm:collapse}
Let  $\descript$ be a description over $(\dd{f}_1, \dots, \dd{f}_n, \dd{g}_1, \dots, \dd{g}_m)$ in a reduced form. 

Suppose first that $\descript$ does not contain a monotone symbol. 
Let $l= \lceil \chrd/2 \rceil$. 
\begin{itemize}
    \item If $m=0$ and $\chrd$ is odd, then $\Clo(\descript) \sim \minAl$.
    \item If $m=0$ and $\chrd$ is even, then $\Clo(\descript) \sim \minBl$.
    \item If $m>0$ and $\chrd$ is even, then $\Clo(\descript) \sim \minClp$.
    \item If $m>0$ and $\chrd$ is odd, then  $\Clo(\descript) \sim \minDl$.
\end{itemize}

Suppose now that $\descript$ contains a monotone symbol. 
\begin{itemize}
    \item If $m=0$, then $\Clo(\descript) \sim \minBinf$.
    \item If $m>0$ and no $\dd{g}_j$ is monotone, then $\Clo(\descript) \sim \minCinf$.
    \item If $m>0$ and some $\dd{g}_j$ is monotone, then $\Clo(\descript) \sim \minDinf$.
\end{itemize}

\end{theorem}

\begin{proof}
    In each of the cases we need to show that $\Clo(\descript) \sim \minion{M}$ for the corresponding minion $\minion M$. Let $k$ be the number of sorts of $\minion{M}$, that is, $k=l$ for $\minAl$, $\minBl$, and $\minDl$, $k=l+1$ for $\minClp$, $k=1$ for $\minBinf$ and $\minDinf$, and $k=2$ for $\minCinf$. 
        
    In the the first part we prove the inequality $\Clo(\descript) \leq \minion{M}$.
    We define a mapping $\xi \colon \Clo(\descript) \to \minion{M}$ (more precisely, a collection $\xi = (\xi^{(j)})_{j \in \mathbb{N}}$)  by
$$
\xi ((f_1, \dots, f_n, g_1, \dots, g_m)) = (h_1, h_2, \dots, h_k),
$$ 
where each $h_i$, which we specify later, can be defined by a term over the set of symbols $\dd{f}_1$, \dots, $\dd{f}_n$, $\dd{g}_1$, \dots, $\dd{g}_m$ in the signature $\{\wedge,\vee, {}^d\}$ as in \Cref{lem:preserving-minors} (the terms can be chosen so that they only depend on $\descript$ and not on the $f_i$ and $g_i$ or their arity); we will say that the $h_i$ are \emph{term-defined}. The lemma then guarantees that $\xi$ preserves minors. 

Moreover, we will choose the $h_i$ so that the definition of $\xi$ makes sense, that is, $(h_1, \dots, h_k)$ is always in $\minion M$. This will show that $\xi$ is indeed a well-defined minion homomorphisms, concluding the first part. 

We first deal with the case that $\descript$ does not contain a monotone symbol. 
The following claim is crucial. 

\begin{claim} \label{claim:chain}
There are term-defined $t_1, \dots,  t_l, s_1, \dots, s_l$ such that
    $$
    s_1 \triangleleft s_2 \triangleleft \dots \triangleleft s_l \leq t_l^d \triangleleft \dots \triangleleft t_2^d \triangleleft t_1^d
    $$
    and, moreover, if $\chrd$ is even, then $s_l \triangleleft t_l^d$. 
\end{claim}
\begin{proof}
    We distinguish cases following the definition of $\chrd$. We remark that more cases can happen simultaneously, in which case we can select any of them. 
    \begin{itemize}
        \item $\chrd = \rank(\dd{f}_i)$ for some $i$. We denote $r=\rank(\dd{f}_i)$, take a path $\dd{f}_{i_1} \triangleleft \dd{f}_{i_2} \triangleleft \dots$ $\triangleleft \dd{f}_{i_r}$  to $\dd{f}_{i}$ in $\descript$ (i.e., $i_r = i$), and set
        \begin{align*}
            (s_1, \dots, s_l) &= (f_{i_1}, f_{i_2}, \dots, f_{i_l}) \\
            (t_1, \dots, t_l) &= (f_{i_r}^d, f_{i_{r-1}}^d, \dots, f_{i_{r-l+1}}^d).
        \end{align*}
        For even $r$, we have $r=2l$, so $r-l+1=l+1$ and then $s_1 \triangleleft \dots \triangleleft s_l = f_{i_l} \triangleleft f_{i_{l+1}} = t_l^d \triangleleft \dots \triangleleft t_1^d$. For odd $r$, we have $r=2l-1$, so $r-l+1=l$ and we have $s_l=t_l^d$. 

        \item $\chrd = \rank(\dd{f}_i) + \rank(\dd{f}_j)$ for some $i,j$ such that  $\dd{f}_i \triangleleft \dd{f}_j^d$ is in  $\descript$. We denote $r = \rank(\dd{f}_i)$, $r' = \rank(\dd{f}_j)$. Note that we have both $f_i \triangleleft f_j^d$ and $f_j \triangleleft f_i^d$ by skew symmetry, therefore we may assume $r \geq r'$. Take paths $\dd{f}_{i_1} \triangleleft \dots \triangleleft \dd{f}_{i_r}$ to $\dd{f}_i$ and $\dd{f}_{j_1} \triangleleft \dots \triangleleft \dd{f}_{j_{r'}}$ to $\dd{f}_j$. Set
        \begin{align*}
            (s_1, \dots, s_l) &= (f_{i_1}, f_{i_2}, \dots, f_{i_l}) \\
            (t_1, \dots, t_l) &= (\underbrace{f_{j_1}, \dots, f_{j_{r'}}}_{r'}, \underbrace{f_{i_r}^d, f_{i_{r-1}}^d, \dots, f_{i_{r+r'-l+1}}^d)}_{l-r'}.
        \end{align*}
        The conditions are easily verified noting that if $r=r'$, then $s_l = f_{i_r}$ and $t_l = f_{j_{r'}}$, and if $r>r'$, then $t_l = f_{i_{l+1}}^d$ if $r+r'$ is even and $t_l=f_{i_l}^d$ if $r+r'$ is odd.
        
        \item $\chrd = 2\rank(\dd{f}_i)+1$ for some $i \in [n]$ such that $\dd{f}_i \triangleleft \dd{g}_j$ is in $\descript$ for some $j \in [m]$. In this case $\rank(\dd{f}_i)=l-1$. We take a path $\dd{f}_{i_1} \triangleleft \dots \dd{f}_{i_{l-1}}$ to $\dd{f}_i$ and set
        \begin{align*}
            (s_1, \dots, s_l) &= (f_{i_1}, f_{i_2}, \dots, f_{i_{l-1}}, g_j) \\
            (t_1, \dots, t_l) &= (f_{i_1}, f_{i_2}, \dots, f_{i_{l-1}}, g_j) 
        \end{align*}
    \end{itemize}
\end{proof}  

We take $s_i$ and $t_i$ as in the previous claim and derive some consequences of their properties. We denote 
$$
u = s_l \wedge t_l.
$$

\begin{claim} \label{claim:prop}
We have $s_1 \wedge t_1 \triangleleft \dots \triangleleft s_{l-1} \wedge t_{l-1} \triangleleft s_l \wedge t_l = u \leq u^d$. If $s_l \triangleleft t_l^d$ (which is the case when $\chrd$ is even), then $u \triangleleft u^d$. 

Moreover, if $g$ is an arbitrary Boolean operation of the same arity as $u$ (and the $s_i$, $t_i$) such that $g=g^d$, then $u \leq (g \vee u) \wedge u^d = ((g \vee u) \wedge u^d)^d$.
\end{claim}

\begin{proof}
  Since $s_l \leq t_l^d \triangleleft \dots \triangleleft t_1^d$ we also have $t_1 \triangleleft \dots \triangleleft t_l \leq s_l^d$ by skew symmetry. Now $s_1 \triangleleft \dots \triangleleft s_l \leq t_l^d$ and $t_1 \triangleleft \dots \triangleleft t_l \leq s_l^d$. By compatibility of $\leq$ and $\triangleleft$ with $\wedge$, we obtain
  $$
  s_1 \wedge t_1 \triangleleft \dots \triangleleft s_{l-1} \wedge t_{l-1} \triangleleft s_l \wedge t_l \leq t_l^d \wedge s_l^d \leq t_l^d \vee s_l^d = (s_l \wedge t_l)^d = u^d.
  $$
  If $s_l \triangleleft t_l^d$, then $t_l \triangleleft s_l^d$, therefore
  $
  u = s_l \wedge t_l \triangleleft t_l^d \wedge s_l^d \leq u^d,
  $
  so $u \triangleleft u^d$.

  As for the second part, we already know $u \leq u^d$  and we assume $g=g^d$, so the claim follows from \Cref{lem:h}.
\end{proof}

It is now time to define the $h_i$. If $m=0$ we set
$$
(h_1, h_2, \dots, h_{k}) 
    =(s_1 \wedge t_1, s_2 \wedge t_2, \dots, s_{k} \wedge t_{k})
$$
and if $m>0$ we pick, say, $g=g_1$, and set
$$
   (h_1, h_2, \dots, h_{k}) 
    =(s_1 \wedge t_1, s_2 \wedge t_2, \dots, s_{k-1} \wedge t_{k-1}, 
       (g \vee u) \wedge  u^d)
$$
In order to verify $(h_1, \dots, h_k) \in \minion{M}$, it is enough to show the following three properties by the definition of the involved minions. 
\begin{itemize}
    \item We have $h_1 \triangleleft h_2 \triangleleft \dots \triangleleft h_l \leq h_l^d$.
    \item If $\chrd$ is even, then $h_l \triangleleft h_l^d$.
    \item If $m>0$, then $h_l \leq h_k = h_k^d$. 
\end{itemize}
All these properties follow from \Cref{claim:prop}.
This finishes the case that $\descript$ does not contain a monotone symbol.

Assume now that $\descript$ contains a monotone symbol.
If some symbol $\dd{g}_j$ is monotone (so $\minion{M} = \minDinf$) we simply define $(h)=(g_j)$. Otherwise we reuse \Cref{claim:prop} as follows. Let $\dd{f}_i$ be a monotone symbol and set $l=2$, $(s_1,s_2)=(f_i,f_i)$, $(t_1,t_2) = (f_i^d,f_i^d)$, and $u = s_l \wedge t_l$ (as before). The claim implies $u \triangleleft u \leq u^d$, so $u \triangleleft u \triangleleft u^d$. If $m=0$ (so $\minion{M} = \minBinf$) we can thus define $(h)=(u)$. If $m>0$ (so $\minion{M} = \minCinf$) we take $g=g_1$ and the claim shows that $(h_1,h_2) = (u, (g \vee u) \wedge u^d)$ works.
    
In the second part of the proof we define a homomorphism $\zeta: \minion{M} \to \Clo(\descript)$ witnessing $\minion{M} \leq \Clo(\descript)$ by
$$
\zeta ((h_1, \dots, h_k)) = (f_1, \dots, f_n, g_1, \dots, g_m),
$$ 
where the $f_i$ and $g_i$ will again be term-defined. Such a $\zeta$ preserves minors 
so we only need to be careful to define the $f_i$ and $g_i$ so that they are idempotent and $(f_1, \dots, f_n, g_1, \dots, g_m)$ satisfies all the constraints in $\descript$. 

If $\descript$ contains a monotone symbol, then we set $\zeta (h) = (h,h, \dots, h)$ if $m=0$ (so $\minion{M}=\minBinf$) or some  $\dd{g}_j$ is monotone (so $\minion{M}=\minDinf$), and $\zeta (h_1,h_2) = (h_1, \dots, h_1, h_2, \dots, h_2)$ for $\minion{M}=\minCinf$. The satisfaction of all the constraints is immediate in these cases.

Suppose further that $\descript$ does not contain a monotone symbol. 
We denote $r = \chrd$ (so $r=2l$ or $r=2l-1$) and $r_i = \rank(\dd{f}_i)$ for $i \in [n]$. Note that $r_i \leq r$ by the definition of $\chrd$. We set 
\begin{align*}
(h'_1, h'_2, \dots, h'_{r})
&= (h_1, h_2, \dots, h_l, h_{r-l}^d, h_{r-l-1}^d, \dots, h_1^d) \\
    (f_1, \dots, f_n, g_1, \dots, g_m) &= 
    (\underbrace{h'_{r_1}, h'_{r_2}, \dots, h'_{r_n}}_{n \text{ operations }},
    \underbrace{h_k, h_k, \dots, h_k}_{m \text{ operations }}).
\end{align*}
We need to verify that all the constraints are satisfied.
\begin{itemize}
    \item $\dd{f}_i \triangleleft \dd{f}_j$ in $\descript$. We need to verify $h'_{r_i} \triangleleft h'_{r_j}$, which is true since $r_i < r_j$ by the definition of rank (as observed after \Cref{def:rank}) and $h'_1 \triangleleft \dots \triangleleft h'_r$ in all the four cases: Indeed, if $r=2l$ (so $\minion{M}$ is $\minBl$ or $\minClp$), then $h_1 \triangleleft \dots \triangleleft h_l \triangleleft h_l^d=h^d_{r-l}$ and the rest is by skew symmetry. If $r=2l-1$, then $h_1 \triangleleft \dots \triangleleft h_l \triangleleft h_{l-1}^d=h^d_{r-l}$, where the last inequality follows from $h_l \leq h_l^d \triangleleft h_{l-1}^d$.

    \item $\dd{f}_i \triangleleft \dd{f}_j^d$ in $\descript$. We need to verify $h'_{r_i} \triangleleft (h'_{r_j})^d$. We can assume $r_i \leq r_j$ by skew symmetry. We have $r_i + r_j \leq r \leq 2l$ by the definition of $\chrd$. This yields two cases: either both  $r_i$ and $r_j$ are less than or equal to $l$, or $r_i \leq l$ and $r_j > l$.
    In the first case, we have $h'_{r_j} = h_{r_j} = (h'_{r-r_j+1})^d$ so
    $h'_{r_i} \triangleleft (h'_{r_j})^d = h'_{r-r_j+1}$ follows from $r_i < r-r_j+1$ and $h'_1 \triangleleft \dots \triangleleft h'_r$ (see the first item).
    In the second case 
    $h'_{r_i} = h_{r_i}$ and $h'_{r_j} = h_{r-r_j+1}^d$. We thus need to confirm that $h_{r_i} \triangleleft h_{r-r_j+1}$, which is the case since $r_i < r-r_j+1$. 

    \item $\dd{f}_i \triangleleft \dd{g}_j$ in $\descript$. We need to verify $h'_{r_i} \triangleleft h_k$. We have $2r_i + 1 \leq r$ by the definition of $\chrd$, so $r_i \leq l-1$ both if $r=2l-1$ ($\minion{M}=\minDl$) or $r=2l$ ($\minion{M}=\minClp$). The relation $h'_{r_i} \triangleleft h_k$ is satisfied in both cases.

    \item $\dd{g}_j = \dd{g}_j^d$ in $\descript$. We need to verify $h_k = h_k^d$, which is true both in $\minDl$ and $\minClp$.
\end{itemize}
This case analysis concludes the proof of \Cref{thm:collapse}.
\end{proof}

\subsection{Minion homomorphisms} \label{subsec:homo}

In this subsection, we almost finish the classification by proving that the inequalities between the minions introduced in the previous subsection, depicted in \Cref{fig:order0,fig:order}, are satisfied, that some inequalities are not satisfied, and that all the minions introduced in \Cref{def:thecores} are cores. From these facts, the main result of the paper will easily follow.

The first step is simple.

\begin{theorem} \label{thm:inequalities}
    The following inequalities are satisfied.
    \begin{align*}
    & \minDinf \leq \minCinf \leq \minBinf \\
    &\minBinf \leq \dots  \leq \minBx3 \leq \minAx3 \leq \minBx2 \leq \minAx2 \leq \minBx1 \leq \minAx1 \leq \minT \\
    &\minCinf \leq \dots  \leq \minCx4 \leq \minDx3 \leq \minCx3 \leq \minDx2 \leq \minCx2 \leq \minDx1 \\
     &\minCx{k+1} \leq \minBk,  \quad   \minDk \leq \minAk \quad \forall k \in \mathbb{N}
    \end{align*}
\end{theorem}

\begin{proof}
    The inequalities are witnessed by homomorphisms in \Cref{table:homomorphisms}. 
    To see that the image of a domain operation belongs to the codomain recall that every $\tuple{h} = (h_1, \dots, h_k)$ in any of the considered minions satisfies $h_i \triangleleft h_i^d$ for every $i \in [k-1]$. 
\end{proof}    
\begin{table}[ht] 
\centering
\begin{tabular}{ll}
$\minDinf \rightarrow \minCinf $ & $(h) \mapsto (h, h)$ \\
$\minCinf \rightarrow \minBinf $ & $(h_1, h_2) \mapsto (h_1)$ \\
$\minBinf \rightarrow \minBk$ & $(h) \mapsto (\underbrace{h, \dots, h}_{k\ \text{times}})$ \\
$\minBk\rightarrow \minAk$ & inclusion \\
$\minAx{k+1} \rightarrow \minBk$ & $(h_1, \dots, h_k, h_{k+1}) \mapsto (h_1, \dots, h_k)$ \\
$\minCinf \rightarrow \minCk$ & $(h_1, h_2) \mapsto (\underbrace{h_1, \dots, h_1}_{k-1\ \text{times}}, h_2)$ \\
$\minCx{k+1} \rightarrow \minDk$ & $(h_1, \dots, h_k, h_{k+1}) \mapsto (h_1, \dots, h_{k-1}, h_{k+1})$ \\
$\minDx{k} \rightarrow \minCx{k}$ & inclusion \\
$\minCx{k+1} \rightarrow \minBk$ & $(h_1, \dots, h_k, h_{k+1}) \mapsto (h_1, \dots, h_k)$ \\
$\minDk \rightarrow \minAk$ & inclusion  \\
\end{tabular}
\caption{Minion homomorphisms witnessing inequalities}
\label{table:homomorphisms}
\end{table}

The other two tasks will be approached simultaneously. 

The multisorted function minions that we consider are Boolean. \Cref{lem:binaries,lem:cores-binaries} then imply that any homomorphism $\xi$ between such minions is fully determined by $\xi^{(2)}$ and, moreover, $\minion{M}$ is a minion core whenever $\xi^{(2)}$ is a bijection $\minion{M}^{(2)} \to \minion{M}^{(2)}$ for every homomorphism $\xi:\minion{M}\to\minion{M}$. Binary operations and binary multisorted operations are therefore important for us. We simplify the notation and write
$$
x = \pi^2_1, \quad y = \pi^2_2,
\quad
\tuple{x} = \tuple{\pi}^2_1 = (x,x, 
\dots, x), \quad
\tuple{y} = \tuple{\pi}^2_2 = (y,y, \dots, y). 
$$
We also often shorten the notation for tuples and write, e.g., $xxxyy$ instead of $(x,x,x,y,y)$.

Recall that every multisorted operation $\tuple{h}=(h_1, \dots, h_k)$ in any of the considered minions satisfies
$$
h_1 \triangleleft \dots \triangleleft h_{k-1} \leq h_k \leq h_k^d \leq h_{k-1}^d \triangleleft \dots \triangleleft  h_1^d, \quad
(\forall i \in [k-1]) \ h_i \triangleleft h_i^d
$$
In particular, if $\tuple{h}$ is binary, then 
each $h_i$ is either $x$, $y$, or $\wedge$, never $\vee$ since $\vee \not\leq \vee^d = \wedge$. In fact, $\wedge \leq x,y \leq \vee$ (and $\wedge \leq \vee$) are all the inequalities that hold between the Boolean idempotent operations. It follows that every binary $\tuple{h}$ is of the form
\begin{align*}
 \wedge^ix^{k-i} &= (\underbrace{\wedge,  \dots, \wedge}_\text{$i$ times}, \underbrace{x,  \dots, x}_\text{$k-i$ times}) \qquad \text{ or }\\
 \wedge^iy^{k-i} &= (\underbrace{\wedge,  \dots, \wedge}_\text{$i$ times}, \underbrace{y,  \dots, y}_\text{$k-i$ times})
\end{align*}
for some $i \in \{0,1, \dots, k\}$.

Our minions are also idempotent, therefore every homomorphism $\xi$ between them satisfies $\xi^{(2)}(\tuple x)=\tuple x$ and $\xi^{(2)}(\tuple y)=\tuple y$ by \Cref{lem:cores-projections}. 

These considerations already imply the following result.

\begin{lemma}\label{lem:Dinf}
    Minion $\minDinf = \{(h) \in \idemp_1 \mid h \triangleleft h = h^d\}$ is a minion core.
\end{lemma}

\begin{proof}
The binary part of $\minDinf$ is $\minDinf^{(2)} = \{x,y\}$ since $\wedge$ does not satisfy $\wedge = \wedge^d$. By \Cref{lem:cores-projections}, every homomorphism $\xi: \minDinf \to \minDinf$ satisfies $\xi(x)=x$ and $\xi(y)=y$, therefore $\xi^{(2)}$ is the identity on $\minDinf^{(2)}$. By \Cref{lem:cores-binaries}, $\minDinf$ is a minion core.
\end{proof}

In the following result about $\minBinf$ we additionally use that minion homomorphisms map commutative operations to commutative operations.

\begin{lemma} \label{lem:Binf}
Minion $\minBinf = \{(h) \in \idemp_1 \mid h \triangleleft h \triangleleft h^d\}$ is a minion core. Moreover, $\minBinf \not\leq \minDx{1}$. 
\end{lemma}

\begin{proof}
The binary part of $\minBinf$ is $\minBinf^{(2)} = \{x,y,\wedge\}$.
As $\wedge \circ xy= \wedge \circ yx$, every minion homomorphism  $\xi$ from $\minBinf$ to a multisorted Boolean minion satisfies $\xi^{(2)}(\wedge) \circ \tuple{xy} = \xi^{(2)}(\wedge) \circ \tuple{yx}$. The only such operation in $\minBinf^{(2)}$ is $\wedge$, therefore every minion homomorphism $\xi: \minBinf \to \minBinf$ satisfies $\xi^{(2)}(\wedge)=\wedge$, and also $\xi^{(2)}(x) = x$ and $\xi^{(2)}(y)=y$  by \Cref{lem:cores-projections}, so $\minBinf$ is a minion core by \Cref{lem:cores-binaries}. 
Moreover, $\minDx{1} = \{(h) \in \idemp_1 \mid h = h^d \}$ does not contain such an operation, so there is no minion homomorphism $\minBinf \to \minDx{1}$.
\end{proof}

Homomorphism from the remaining minions will be studied using 5-ary symmetric operations, that is, operations whose result does not change when the arguments are permuted. Note that a $k$-sorted Boolean 5-ary operation $\tuple{h}$ is symmetric if, and only if,
$$
\tuple{h} \circ (\tuple{\pi}^5_1, \tuple{\pi}^5_2, \tuple{\pi}^5_3, \tuple{\pi}^5_4, \tuple{\pi}^5_5) =
\tuple{h} \circ (\tuple{\pi}^5_{\sigma(1)}, \tuple{\pi}^5_{\sigma(2)}, \tuple{\pi}^5_{\sigma(3)}, \tuple{\pi}^5_{\sigma(4)}, \tuple{\pi}^5_{\sigma(5)}), 
$$
for every permutation $\sigma$  on $[5]$. It follows that every minion homomorphism maps 5-ary symmetric multisorted operations to 5-ary symmetric multisorted operations.

The following observation makes it explicit how relations involving $\leq$, $\triangleleft$, and ${}^d$ are decided for such operations. We omit the proof.

\begin{lemma} \label{lem:symmetric-5ary}
  Let $h$ and $h'$ be 5-ary symmetric Boolean idempotent operations and denote $(a_1,a_2,a_3,a_4) = (h(10000),h(11000),h(11100),h(11110))$ and similarly $(a_1',a_2',a_3',a_4')$ (note that these tuples determine the operations by symmetry and idempotency).

 \begin{center}
\begin{tabular}{l|l}
  & $h$  \\ \hline
10000 & $a_1$  \\
11000 & $a_2$  \\
11100 & $a_3$  \\
11110 & $a_4$  \\  
\end{tabular}
\qquad
\begin{tabular}{l|l}
  & $h'$  \\ \hline
10000 & $a'_1$  \\
11000 & $a'_2$  \\
11100 & $a'_3$  \\
11110 & $a'_4$  \\  
\end{tabular}
\end{center}

  \begin{itemize}
      \item $h \leq h'$ iff $a_i \leq a'_i$ for all $i \in [4]$.
      \item $h \triangleleft h'$ iff $a_i \leq a'_j$ for all $i \leq j$, $i,j \in [4]$. 
      \item $h \leq h^d$ iff $(a_1,a_4) \neq (1,1)$ and $(a_2,a_3) \neq (1,1)$.
      \item $h \triangleleft h^d$ iff $a_1=a_2=0$.
      \item $h = h^d$ iff $a_1 = \overline{a_4}$ and $a_2 = \overline{a_3}$.
  \end{itemize}
\end{lemma}

The last minion with infinite subscript is dealt with in the following lemma. A similar proof strategy will be used for the remaining cases.

\begin{lemma}\label{lem:Cinf}
    Minion $\minCinf = \{ (h_1, h_2) \in \idemp_2 \mid h_1 \triangleleft h_1 \triangleleft h_2 = h_2^d\}$ is a minion core. Moreover, $\minCinf \not\leq \minDinf$.
\end{lemma}

\begin{proof}
    Let $\tuple t$ and $\tuple{t}'$ be the following 2-sorted 5-ary symmetric Boolean operations.
 \begin{center}
\begin{tabular}{l|ll}
$\tuple t$  & $t_1$ & $t_2$  \\ \hline
10000 & 0    & 0      \\
11000 & 0    & 1       \\
11100 & 0    & 0        \\
11110 & 0    & 1   \\  
\end{tabular}
\qquad
\begin{tabular}{l|ll}
$\tuple t'$  & $t'_1$ & $t'_2$  \\ \hline
10000 & 0    & 1      \\
11000 & 0    & 0       \\
11100 & 0    & 1        \\
11110 & 0    & 0   \\  
\end{tabular}
\end{center}
Observe that both of them are in $\minCinf$ and they satisfy $\tuple{t}\xxxxy = \tuple{t} \xxyyy$ and $\tuple{t}'\xxxxy = \tuple{t}'\xxyyy$ (which exactly means that the 10000 row is equal to the 11100 row and the 11000 row is equal to the 11110 row).

We claim that, conversely, if $\tuple{h} \in \minCinf^{(5)}$ is symmetric and satisfies $\tuple{h}\xxxxy = \tuple{h}\xxyyy$, then necessarily $\tuple{h}=\tuple{t}$ or $\tuple{h}=\tuple{t}'$. Indeed, $h_1 \triangleleft h_1^d$ implies $h_1(10000)=h_1(11000)=0$. Then $h_1(xxxxy) = h_1(xxyyy)$ and symmetry imply that $h_1(11110)$ and $h_1(11100)=h_1(00111)$ are both 0 as well. Moreover, $h_2(10000)=h_2(11100)$ and $h_2(11000)=h_2(11110)$  by  $h_2(xxxxy) = h_2(xxyyy)$ (plus symmetry), and $h_2(10000)=\overline{h_2(11110)}$ and $h_2(11000) = \overline{h_2(11100)}$ by $h_2=h_2^d$, which gives us $\tuple{h}=\tuple{t}$ or $\tuple{h}=\tuple{t}'$, as required.
    
In order to apply \Cref{lem:cores-binaries} to prove that $\minCinf$ is a minion core, we first note that     
$\minCinf^{(2)} = \{xx, yy, \wedge x, \wedge y\}.$
Consider a minion homomorphism $\xi: \minCinf \to \minCinf$. 
By~\Cref{lem:cores-projections}, we know that $\xi^{(2)}(xx)=xx$ and $\xi^{(2)}(yy) = yy$. 
We also know that $\tuple{h} = \xi(\tuple{t})$ is symmetric and $\tuple{h}\xxxxy = \tuple{h}\xxyyy$, therefore $\tuple{h}$ is either $\tuple{t}$ or $\tuple{t}'$. 
In the first case, we obtain $\xi^{(2)}(\wedge x) = \xi^{(2)}( \tuple{t} \circ \tuple{xxxxy}) = \tuple{h} \circ \tuple{xxxxy} = \wedge x$ and $\xi^{(2)}(\wedge y) = \tuple{h} \circ \tuple{yyyyx} = \wedge y$. In the second case, we similarly obtain $\xi^{(2)}(\wedge x)= \wedge y$ and $\xi^{(2)}(\wedge y) = \wedge x$. In both cases, we see that $\xi^{(2)}$ is a bijection, therefore $\minCinf$ is a minion core by \Cref{lem:cores-binaries}.

For the second claim, we just observe that $\minDinf = \{(h) \mid h \triangleleft h = h^d\}$ does not contain any symmetric 5-ary $(h)$ satisfying $h(xxxxy) = h(xxyyy)$.
\end{proof}

It remains to deal with $\minAk$, $\minBk$, $\minCk$, and $\minDk$.
We fix $k$ and define $k$-sorted 5-ary symmetric idempotent Boolean operations for $i \in [k]$ as follows. It could be helpful to also look at \Cref{fig:tables}.
\begin{align*}
\tuple t^{i} \xxxxy &= \wedge^ix^{k-i}, \quad \tuple{t}^i \xxxyy = \wedge^{i-1}x^{k-i+1} \\
\tuple u \xxxxy &= \wedge^{k-1}y^1, \quad \tuple u \xxxyy = \wedge^{k-2}x^2  \\
\tuple v \xxxxy &= \wedge^{k-1}y^1, \quad \tuple v \xxxyy = \wedge^{k-1}x^1 \\
\tuple w \xxxxy &= \wedge^{k-1}x^1, \quad \tuple w \xxxyy = \wedge^{k-1}y^1
\end{align*}

\begin{figure}
\begin{center}
\begin{tabular}{l|lllll}
$\tuple t^{1}$  & $t^1_1$ & $t^1_2$ & $t^1_3$  & \dots    & $t^1_k$ \\ \hline
10000 & 0    & 0    & 0 & \dots &  0    \\
11000 & 0    & 0    & 0 & \dots &  0   \\
11100 & 1    & 1    & 1 & \dots &  1     \\
11110 & 0    & 1    & 1 & \dots &  1  
\end{tabular}
\qquad
\begin{tabular}{l|llllll}
$\tuple t^{2}$  & $t^2_1$ & $t^2_2$ & $t^2_3$  & \dots  & $t^2_k$ \\ \hline
 & 0    & 0    & 0 & \dots  &  0    \\
 & 0    & 0    & 0 & \dots  &  0   \\
 & 0    & 1    & 1 & \dots &  1     \\
 & 0    & 0    & 1 & \dots &  1  
\end{tabular}
\end{center}

\begin{center}
\begin{tabular}{l|ccccc}
$\tuple t^{k-1}$  & $t^{k-1}_1$ &  &  & $t^{k-1}_{k-1}\!\!$ & $t^{k-1}_k$ \\ \hline
10000 & 0    & \dots &  0   & 0          & 0    \\
11000 & 0    & \dots &  0   & 0          & 0   \\
11100 & 0    & \dots &  0  & 1          & 1     \\
11110 & 0    & \dots &  0  & 0          & 1  
\end{tabular}
\qquad
\begin{tabular}{l|ccccc}
$\tuple u$  & $u_1$ &  &  & $u_{k-1}$ & $u_k$ \\ \hline
 & 0    & \dots &  0   & 0          & 1    \\
 & 0    & \dots &  0   & 0          & 0   \\
 & 0    & \dots &  0  & 1          & 1     \\
 & 0    & \dots &  0  & 0          & 0  
\end{tabular}
\end{center}

\begin{center}
\begin{tabular}{l|ccccc}
$\tuple t^{k}$  & $t^k_1$ & \dots &  & $t^k_k$ \\ \hline
10000 & 0  & \dots &  0  & 0    \\
11000 & 0  & \dots &  0  & 0   \\
11100 & 0  & \dots &  0  & 1     \\
11110 & 0  & \dots &  0  & 0   
\end{tabular}
\qquad
\begin{tabular}{l|cccc}
$\tuple v$   & \dots &  & $v_k$ \\ \hline
 & \dots &  0  & 1    \\
 & \dots &  0  & 0   \\
 & \dots &  0  & 1     \\
 & \dots &  0  & 0   
\end{tabular}
\qquad
\begin{tabular}{l|cccc}
$\tuple w$   & \dots &  & $w_k$ \\ \hline
 & \dots &  0  & 0    \\
 & \dots &  0  & 1   \\
 & \dots &  0  & 0     \\
 & \dots &  0  & 1   
\end{tabular}
\end{center}
\caption{Operations $\tuple{t}^i$, $\tuple{u}$, $\tuple{v}$, and $\tuple{w}$.}
\label{fig:tables}
\end{figure}

These operations will play a similar role to operations $\tuple{t}$ and $\tuple{t}'$ in the proof of \Cref{lem:Cinf}.

For convenience, we recall the definitions of the minions we consider. Note in particular that $\minBk \subseteq \minAk$ and $\minDk \subseteq \minCk$.
    \begin{align*}
        \minAk \qquad \qquad &h_1 \triangleleft h_2  \triangleleft \dots \triangleleft h_k \leq h_k^d \\ 
        \minBk \qquad \qquad &h_1 \triangleleft h_2  \triangleleft \dots \triangleleft h_k \triangleleft h_k^d \\
        \minCk \qquad \qquad &h_1 \triangleleft h_2 \triangleleft \dots \triangleleft h_{k-1} \leq h_{k} = h_{k}^d, \ h_{k-1} \triangleleft h_{k-1}^d  \\
        \minDk \qquad\qquad &h_1 \triangleleft h_2  \triangleleft \dots \triangleleft h_k = h_k^d 
   \end{align*}
\begin{lemma}   \label{lem:chain-membership}
    The operations $\tuple{t}^1$, \dots, $\tuple{t}^{k-1}$ are all in $\minBk$ and $\minDk$ (and thus in $\minAk$ and $\minCk$),  $\tuple{t}^{k} \in \minBk$, $\tuple{v}, \tuple{w} \in \minDk$,  and $\tuple{u} \in \minCk$, but $\tuple{v}, \tuple{w} \not\in \minBk$ for any $k$ and $\tuple{u} \not\in \minAk, \minDk$ for $k \geq 2$.
\end{lemma}
\begin{proof}
    The proof is straightforward using \Cref{lem:symmetric-5ary}; the  non-membership  for $\tuple{u}$ is because $u_{k-1}(11100)=1$ but $u_k(11110)=0$, so $u_{k-1} \ntriangleleft u_k$. The non-membership for $\tuple{v}$ is because $v_k(10000)=1$, so $\tuple{v} \ntriangleleft \tuple{v}^d$, and similarly for $\tuple{w}$. 
\end{proof}

The equations we use are as follows.

\begin{itemize}
    \item[(Chain)] \label{id:chain} $\tuple h^1 \xxxyy = \tuple{x}$,
        $\tuple h^{i+1} \xxxyy = \tuple h^i \xxxxy \quad \forall i\in[k-1]$
    \item[(AB)] \label{id:ab} $\tuple h^k \xxxxy =\tuple h^k \yyyyx$ 
    \item[(CD)] \label{id:cd} $\tuple h^k \xxxxy =\tuple h^k \xxyyy$
\end{itemize}

\begin{lemma} \  \label{lem:chain-equations}

\begin{itemize}    
\item The operations $(\tuple{h}^1, \dots, \tuple h^k) =(\tuple{t}^1, \dots, \tuple{t}^{k-2}, \tuple{t}^{k-1}, \tuple{t}^k)$ satisfy (Chain) and (AB).

\item The operations $(\tuple{h}^1, \dots, \tuple h^k) = (\tuple{t}^1, \dots, \tuple{t}^{k-2}, \tuple{t}^{k-1}, \tuple{v})$ and  $(\tuple{h}^1, \dots, \tuple h^k) =(\tuple{t}^1, \dots, \tuple{t}^{k-2}, \tuple{u}, \tuple{w})$ both satisfy (Chain) and (CD).
\end{itemize}
\end{lemma}
\begin{proof}
    This is an immediate consequence of the definition of these operations.
\end{proof}

The crucial lemma is the following partial converse to the previous two lemmas.

\begin{lemma} \label{lem:chain-forcing}
  Let $\tuple{h}^1$, \dots, $\tuple{h}^k$ be $k$-sorted operations, all in $\minAk$ or all in $\minCk$. 
  \begin{itemize}
      \item If the operations satisfy (Chain) and (AB), and they are all in $\minAk$, then 
      $(\tuple{h}^1, \dots, \tuple h^k)=(\tuple{t}^1, \dots, \tuple{t}^k)$.
      \item If the operations satisfy (Chain) and (CD), then $(\tuple{h}^1, \dots, \tuple h^k)$ is equal to $(\tuple{t}^1, \dots, \tuple{t}^{k-2}, \tuple{t}^{k-1}, \tuple{v})$ or  $(\tuple{t}^1, \dots, \tuple{t}^{k-2}, \tuple{u}, \tuple{w})$.
  \end{itemize}
\end{lemma}

\begin{proof}
   The first claim is common for both items in the lemma. It is a consequence of (Chain).

    \begin{claim} \label{cl:propagate_one}
        For every $i \in [k-1]$ and $j \in [k]$ with $j \geq i$, we have $h^i_j(xxxyy) = x$. Moreover, if $\tuple{h}^{k-1}$ is in $\minAk$ or $\minDk$, then also $h^k_k(xxxyy)=x$.
    \end{claim}
    \begin{proof}
        We prove the claim by induction on $i$ starting with $i=1$, where the claim follows immediately from $\tuple{h}^1\xxxyy = \tuple x$. Assume now $1<i \leq k-1$. By the induction hypothesis, we have $h^{i-1}_{i-1}(11100) = 1$. We also know that $h^{i-1}_{i-1} \triangleleft h^{i-1}_{i}$, therefore $h^{i-1}_i(11110)=1$, and then, as $h^{i-1}_{i} \leq h^{i-1}_{i+1} \leq \dots \leq h^{i-1}_k$, we obtain $h^{i-1}_j(11110)=1$ whenever $i \leq j \leq k$. No $h^{i-1}_j(xxxxy)$ is $\vee$, therefore $h^{i-1}_j(xxxxy)$ is equal to $x$, and so is $h^i_j(xxxyy)=h^{i-1}_j(xxxxy)$ by (Chain). If $\tuple{h}^{k-1}$ is in $\minAk$ or $\minDk$, then $h^{k-1}_{k-1} \triangleleft h^{k-1}_k$, so the argument works for $i=k$ as well.
    \end{proof}

    Now we make a case distinction and prove that $\tuple{h}^k$ is as claimed. 
    \begin{itemize}
        \item Assume that all $\tuple{h}^i$ are in $\minAk$ and $\tuple{h}^k$ satisfies (AB). The first assumption and \Cref{cl:propagate_one} ensures $h^k_k(xxxyy)=x$, while the second assumption implies $\tuple{h}^k \xxxxy = \wedge^k$, in particular $h^k_k(xxxxy)=\wedge$. Finally, as $h^k_k(11110)=0$ and $h^k_i \triangleleft h^k_k$ for every $i \in [k-1]$, the $i$th column in the table of $\tuple{h}^k$ consists of zeros and we get $\tuple{h}^k = \tuple{t}^k$. 
        \item Assume that all $\tuple{h}^i$ are in $\minAk$ and $\tuple{h}^k$ satisfies (CD).
        As in the previous case we get $h^k_k(xxxyy)=x$ and (CD) implies $h^k_k(xxxxy)=y$. We again have $h^k_k(11110)=0$ so all but the last columns are zero and we get $\tuple{h}^k = \tuple{v}$. 
        \item Assume that all $\tuple{h}^i$ are in $\minCk$ and $\tuple{h}^k$ satisfies (CD). That the last two components  of $\tuple{h}$ (i.e., $h^k_{k-1}$, $h^k_k$)  are equal to the last two components of $\tuple{v}$ or $\tuple{w}$ is proved as in \Cref{lem:Cinf}: from $h^k_{k-1} \triangleleft (h^k_{k-1})^d$ and (CD) it follows that the $h^k_{k-1}$ column is zero, and from $h^k_k = (h^k_k)^d$ and (CD) it follows that $h^k_k = v_k$ or $h^k_k = w_k$. Then $\tuple{h}=\tuple{v}$ or $\tuple{h} = \tuple{w}$ as $h^k_{k-1}(11110)=0$ and $h^k_i \triangleleft h^k_{k-1}$ for all $i \in [k-2]$. (In this case we did not need the claim yet.)
    \end{itemize}

   The proof is finished by induction going down from $\tuple{h}^{i+1}$ to $\tuple{h}^{i}$ for $i = k-1, k-2, \dots, 1$: The minor $\tuple{h}^{i} \xxxxy$ is equal to $\tuple{h}^{i+1} \xxxyy$ by (Chain), $h^i_j(xxxyy)=x$ for $j \geq i$ by \Cref{cl:propagate_one}, and $h^i_j(xxxyy)=\wedge$ for $j < i$ as $h^i_i(11110)=0$ and $h^i_j \triangleleft h^i_i$.
\end{proof}

\begin{lemma} \label{lem:ABCDk}
    For every $k \in \mathbb N$, minions $\minAk, \minBk, \minCk, \minDk$ are minion cores. Moreover, $\minDk \not\leq \minBk$ and, for $k \geq 2$, $\minCk \not\leq \minAk$.
\end{lemma}

\begin{proof}
    Consider first a minion homomorphism $\xi: \minAk \to \minAk$. The operations $(\tuple{h}^1, \dots, \tuple{h}^k) = (\tuple{t}^1$, \dots, $\tuple{t}^k)$ are in $\minAk$ by \Cref{lem:chain-membership} and they satisfy (Chain) and (AB) by \Cref{lem:chain-equations}. Their $\xi$-images must satisfy the corresponding equations (with $\tuple{h}^i$ replaced by $\xi(\tuple{h}^i)$), so $(\xi(\tuple{h}^1), \dots, \xi(\tuple{h}^k)) = (\tuple{t}^1, \dots, \tuple{t}^k)$ by \Cref{lem:chain-forcing}. 
    It then follows that $\xi^{(2)}(\wedge^ix^{k-i}) = \xi^{(2)}(\tuple{t}^i \xxxxy) = (\xi^{(5)}(\tuple{t}^i)) \xxxxy = \wedge^ix^{k-i}$ and $\xi^{(2)}(\wedge^iy^{k-i})=\wedge^iy^{k-i}$ for $i \in [k]$ and $\xi^{(2)}(\tuple{x})=\tuple{x}$ (e.g., by \Cref{lem:cores-projections}). The binary part of $\xi$ is thus the identity, so $\minAk$ is a minion core by \Cref{lem:cores-binaries}.
    The proof that $\minBk$ is a minion core is the same.

    Consider now the minion $\minDk$. This time we use $(\tuple{h}^1, \dots, \tuple{h}^k) = (\tuple{t}^1, \dots, \tuple{t}^{k-1}, \tuple{v})$ and (Chain) with (CD). \Cref{lem:chain-membership,lem:chain-equations,lem:chain-forcing} again guarantee that the binary part of every minion homomorphism $\minDk \to \minDk$ is the identity, therefore $\minDk$ a minion core. Moreover, there is no minion homomorphism $\minDk \to \minBk$ as neither $\tuple{v}$ nor $\tuple{w}$ is in $\minBk$. 

    To show that $\minCk$ is a minion core we use the same tuple. Our lemmas imply that for every $\xi: \minCk \to \minCk$, its image  $(\xi(\tuple{h}^1), \dots, \xi(\tuple{h}^k))$ is either $(\tuple{t}^1, \dots, \tuple{t}^{k-2}, \tuple{t}^{k-1}, \tuple{v})$ or  $(\tuple{t}^1, \dots, \tuple{t}^{k-2}, \tuple{u}, \tuple{w})$. In the first case, $\xi^{(2)}$ is the identity. In the second case, $\xi^{(2)}(\wedge^{k-1}x) = \wedge^{k-1}y$, $\xi^{(2)}(\wedge^{k-1}y) = \wedge^{k-1}x$, and $\xi^{(2)}$ is identical on the remaining operations in $\minCk^{(2)}$. In both cases, $\xi^{(2)}$ is a bijection, therefore $\minCk$ is a minion core. 

    Finally, to prove that, for $k \geq 2$, there is no minion homomorphism $\xi: \minCk \to \minAk$ we use (Chain) with (CD) and both tuples appearing in the second item of \Cref{lem:chain-equations}. Minion homomorphism $\xi$ would map each of them to one of them. But $\tuple{u} \not\in \minAk$, therefore both tuples would be mapped to $(\tuple{t}^1, \dots, \tuple{t}^{k-2}, \tuple{t}^{k-1}, \tuple{v})$. However,  $\tuple{v} \xxxxy = \tuple{w} \xxxyy$ but $\xi(\tuple{v}) \xxxxy = \tuple{v} \xxxxy \neq \tuple{v} \xxxyy = \xi(\tuple{w}) \xxxyy$, so no such $\xi$ exists. 
\end{proof}

In general, a non-inequality $\minion{M} \not\leq \minion{N}$ can  always be witnessed by height 1 identities satisfiable in $\minion{M}$ but not in $\minion{N}$. Such identities for our minions follow from the proofs above; they are based on symmetry, (Chain), (AB), and (CD). (The minion $\minBinf$ was somewhat exceptional in that we used commutativity. Note however that we could have used (AB).) The most complex is $\minCk \not\leq \minAk$, where an explicit list of identities coming from the proof is the following.
    \begin{align*}
t_i(x_1x_2x_3x_4x_5)&\approx t_i(x_{\sigma(1)}x_{\sigma(2)}x_{\sigma(3)}x_{\sigma(4)}x_{\sigma(5)}) \quad \forall i \in [k] \ \forall \sigma \in S_5 \\   
t'_i(x_1x_2x_3x_4x_5)&\approx t'_i(x_{\sigma(1)}x_{\sigma(2)}x_{\sigma(3)}x_{\sigma(4)}x_{\sigma(5)}) \quad \forall i \in [k] \ \forall \sigma \in S_5 \\   
        t_1(xxxyy) &\approx t_1(xxxxx) \approx t_1'(xxxyy)\\   
        t_{i+1}(xxxyy) &\approx t_i(xxxxy), \quad t_{i+1}'(xxxyy) \approx t_i'(xxxxy) \quad \forall i\in[k-1] \\ 
        t_k(xxxxy) &\approx t_k(xxyyy), \quad t'_k(xxxxy) \approx t_k'(xxyyy) \\
        t_k(xxxxy) &\approx t'_k(xxxyy)
   \end{align*}

An alternative and perhaps more elegant way to achieve our results would be via sequence of ternary operations 
\begin{align*}
t_1(xyy) &\approx t_1(xxx) \\
t_{i+1}(xyy) & \approx t_i(xxy) \quad \forall i\in[k-1]
\end{align*}
together with $t_k(xyx) \approx t_k(yxy)$ instead of (AB) and $t_k(xxy) \approx t_k(yxx)$ instead of (CD). An advantage of symmetric operations is that they are simple to work with, in particular, the relation $\triangleleft$ between them is transparent.

\subsection{Summary} \label{subsec:summary}

We are ready to prove the main result of this paper.

\begin{theorem} \label{thm:main}
Let $\Gamma$ be a set of at most binary relations on a multisorted Boolean set. Then $\Pol(\Gamma)$ is equivalent to exactly one of the minions $\minion{T}$, $\minAk$, $\minBk$, $\minCkp$, $\minDk$, $\minBinf$, $\minCinf$, $\minDinf$, $k \in \mathbb N$. These minions are all minion cores and the ordering by minion homomorphism is exactly as depicted in \Cref{fig:order0,fig:order}.

Conversely, for every minion $\minion{M}$ from the list ($\minion{T}$, $\minAk$, etc.), there exists a set of at most binary relations $\Gamma$ on a multisorted Boolean set such that $\minion{M} \sim \Pol(\Gamma)$. 
\end{theorem}

\begin{proof}
\Cref{lem:description_exists} shows that unless $\Pol(\Gamma) \sim \minion{T}$,  $\Pol(\Gamma)$ is equivalent to $\Clo(\descript)$ for some description $\descript$. By \Cref{thm:reduced-description}, the description can be chosen to be reduced. \Cref{thm:collapse} then implies that $\Pol(\Gamma)$ is equivalent to at least one of the minions from the list. The inequalities between them were verified in \Cref{thm:inequalities}. \Cref{lem:Dinf,lem:Binf,lem:Cinf,lem:ABCDk} then show that all of them are minion cores and that $\minBinf \not\leq \minDx1$, $\minCinf \not\leq \minDinf$, $\minDk \not\leq \minBk$ for all $k \in \mathbb{N}$, and $\minCk \not\leq \minAk$ for all $k \geq 2$.

For the first part of the statement, it remains to observe that there are no inequalities between these minions other than the already established ones. (Note that this also implies the uniqueness.) Assume the converse, that is, there are minions $\minion{M}$ and $\minion{N}$ among them such that $\minion{M} \leq \minion{N}$, but this inequality does not follow from the inequalities established in \Cref{thm:inequalities} by the reflexivity and transitivity of $\leq$. 

We distinguish cases according to $\minion{M}$.
\begin{itemize}
    \item $\minion{M}=\minCinf$. The only $\minion{N}$ such that $\minion{M} \leq \minion{N}$ does not follow from the established inequalities is $\minDinf$, a contradiction with $\minCinf \not\leq \minDinf$.

    \item $\minion{M} = \minBinf$. The only candidates for $\minion{N}$ are $\minDinf$, $\minCinf$, $\minDk$ ($k \geq 1$), and $\minCk$ ($k \geq 2$). But then $\minBinf \leq \minion{N} \leq \minDx1$, a contradiction with $\minBinf \not\leq \minDx1$.

    \item $\minion{M} = \minCk$, $k \geq 2$. The candidates for $\minion{N}$ are $\minAx{l}$, $\minBx{l}$, $\minDx{l}$ for $l \geq k$ (allowing $l=\infty$ from now on), and $\minCx{l}$ for $l > k$. Then $\minCk \leq \minion{N} \leq \minAk$, a contradiction.

    \item $\minion{M} = \minDk$. We similarly get $\minDk \leq \minBk$.

    \item $\minion{M} = \minAk$. The candidates for $\minion{N}$ are $\minBx{l}$ for $l \geq k$ and $\minAx{l}$ for $l > k$, but then $\minDk \leq \minAk \leq \minion{N} \leq \minBk$, and $\minCx{l}$ and $\minDx{l}$, but then $\minBinf \leq \minAk \leq \minion{N} \leq \minDx{1}$. 

    \item $\minion{M} = \minBk$. We similarly get $\minCx{k+1} \leq \minAx{k+1}$ or $\minBinf \leq \minDx{1}$.

    \item $\minion{T}$. $\minion{T} \not\leq \minAx{1}$ since $\pi^2_1$ and $\pi^2_2$ are equal in $\minion{T}$ and different in $\minAx{1}$. 
\end{itemize}

For the converse statement, we just observe that for $\minion{M}=\minion{T}$ one can take e.g. the empty $\Gamma$ on any Boolean multisorted set, and for the other minions we select $\descript$ so that $\Clo(\descript) \sim \minion{M}$ follows from \Cref{thm:collapse}, and then define $\Gamma$ accordingly, including also the singleton unary relations to ensure $\Clo(\descript) = \Pol(\Gamma)$. 
\end{proof}

\section{Function clones and multisorted Boolean clones} \label{sec:small-projections}

The goal of this section is to show that the classification in \Cref{thm:main} is the same for function clones of the form $\Pol(\Psi)$ where $\Psi$ is a set of binary relations with small projections on a finite set. 

It will be convenient to identify $n$-tuples and functions from $[n]$, that is, an $n$-tuple $\tuple{c} \in C^n$ is a function $[n] \to C$. This enables us to write the definition of a minor in a compact way: if $f: C^n \to C$, $\alpha: [n] \to [m]$, and $\tuple{c} \in C^m$, then
$$
f^{(\alpha)}(\tuple{c}) = f(\tuple{c} \circ \alpha).
$$

The following observation will help us to define a clone homomorphism in the subsequent theorem.

\begin{lemma} \label{lem:homo-from-x}
    Let $\minion{M}$ be a minion, let $l \in \mathbb{N}$, and let $\minion{O}$ be the function clone of all operations on $C = [l]$. Then for every function 
    $$X: \minion{M}^{(l)} \to C,$$
    the collection $\xi = (\xi^{(n)}: \minion{M}^{(n)} \to \minion{O}^{(n)})_n$ defined by
    $$
    (\xi^{(n)}(f))(\tuple{c}) = X( \minion{M}^{(\tuple{c})}(f)) \qquad \forall n \in \mathbb{N}, \ \forall f \in \minion{M}^{(n)}, \  \forall \tuple{c} \in C^n
    $$
    is a minion homomorphism from $\minion{M}$ to $\minion{O}$.
\end{lemma}

\begin{proof}
    For $n,m \in \mathbb{N}$, $f \in \minion{M}^{(n)}$, and $\alpha: [n] \to [m]$, we check $\xi^{(m)}(\minion{M}^{(\alpha)}(f)) = \minion{O}^{(\alpha)}(\xi^{(n)}(f))$ by comparing the values of both sides on a tuple $\tuple{c} \in C^m$ (or $\tuple{c}: [m] \to [l]$).
    \begin{align*}
    (\xi^{(m)}(\minion{M}^{(\alpha)}(f)))(\tuple{c})
    &= X( \minion{M}^{(\tuple{c})}(\minion{M}^{(\alpha)}(f)))
    =X(\minion{M}^{(\tuple{c} \circ \alpha)}(f))
    \\
    (\minion{O}^{(\alpha)}(\xi^{(n)}(f)))(\tuple{c})
    &=(\xi^{(n)}(f))^{(\alpha)}(\tuple{c})
    =(\xi^{(n)}(f))(\tuple{c} \circ \alpha)
    =X(\minion{M}^{(\tuple{c} \circ \alpha)}(f))
    \end{align*}
\end{proof}

We remark that in fact every homomorphism $\xi: \minion{M} \to \minion{O}$ arises as in the previous lemma. A natural extension of this fact works for the functions minion of all operations from $C=[l]$ to $D$ (where $X$ would be a function $\minion{M}^{(l)} \to D$), and it can also be naturally extended to the multisorted setting; \Cref{lem:binaries,lem:preserving-minors} could be regarded as special cases of such refinements. Finally, we remark that Lemma~4.4 in~\cite{PCSP} gives a generalization to any polymorphism minion on finite sets instead of $\minion{O}$ by specifying which of the mappings $X$ give rise to a well-defined $\xi$. 

We are ready to show that clones determined by relations with small projections are equivalent to multisorted Boolean clones. Since there is no advantage in restricting to binary relations, we allow arbitrary arities. 

\begin{theorem} \label{thm:translation}
    Let $m \in \mathbb{N}$.
    
    For every set $\Psi$ of at most $m$-ary relations with small projections on a finite set $C$, there exists a set $\Gamma$ of at most $m$-ary relations on a multisorted Boolean set such that  $\Pol(\Psi) \sim \Pol(\Gamma)$.

    Conversely, for every set $\Gamma$ of at most $m$-ary relations on a multisorted Boolean set $\tuple{A}$, there exists a set $\Psi$ of at most $m$-ary relations with small projections on a finite set such that $\Pol(\Gamma) \sim \Pol(\Psi)$.
\end{theorem}

\begin{proof}
  Let $\Psi$ be as in the first part of the statement. We assume $C = [l]$.
 We also assume that $\Psi$ does not contain an empty relation. If $\Psi \sim \minion{T}$, then the claim is obvious, otherwise we may assume that $\Pol(\Psi) = \Polid(\Psi)$ by~\Cref{thm:rel-cores}. 
  Let $S(\Psi)$ denote the set of all projections of relations in $\Psi$ onto their coordinates. 
  
  If $S(\Psi)$ does not contain a two-element set, then $\Polid(\Psi)$ is exactly the clone of all the idempotent operations on $C$. This clone is equivalent to $\idemp_1$ by \cite[Proposition 2.12]{submaximal}%
  \footnote{\cite{submaximal} gives a relational proof of this fact. Alternatively, homomorphisms in both directions can be obtained using \Cref{lem:homo-from-x}, choosing $X$ so that each $\xi(f)$ is idempotent. 
  }
  , which is equal to $\Pol(\Gamma)$, where $\Gamma$ consists of the two singleton unary relations. We now assume that $S(\Psi)$ contains a two-element set.

  We start by defining a multisorted ``essentially Boolean'' set $\tuple{A}$. It will be convenient to index the sorts by selected subsets of $C$ instead of natural numbers. Also, the sorts will not necessarily be equal to $\{0,1\}$, they will be general two-element or one-element sets. Note that the one-element sorts can be safely removed after the construction (adjusting the relations accordingly) and elements renamed, making the multisorted set truly Boolean according to our definition. With these conventions, we define 
  $$
  \tuple{A} = (A_P)_{P \in S(\Psi)}, \qquad \text{where } A_P = P.
  $$
  We continue by defining $\Gamma$. For every relation $R \in \Psi$, say of  arity $r$, we include to $\Gamma$ essentially the same relation $\tilde{R} = R$, but now regarded as a relation on $\tuple{A}$:
  $$
  \tilde{R} \subseteq \pi^r_1(R) \times \pi^r_2(R) \times \dots \times \pi^r_r(R) \qquad
  \text{of type } (\pi^r_1(R), \pi^r_2(R), \dots, \pi^r_r(R)).
  $$
  We also include to $\Gamma$ all singleton unary relations of all types, in particular $\Pol(\Gamma) = \Polid(\Gamma)$. This finishes the definition of $\Gamma$. 
  
  Since each $f \in \Pol(\Psi)$, say of arity $n$, preserves each $P \in S(\Psi)$, it makes sense to define $f_P: P^n \to P$ by restricting the domain of $f$ to $P^n$ and codomain to $P$. The collection $\xi = (\xi^{(n)})_n$ defined by $\xi^{(n)}(f) = (f_P)_{P \in S(\Psi)}$ is then easily seen to be a homomorphism $\Pol(\Psi) \to \Pol(\Gamma)$.

  To prove the other inequality $\Pol(\Gamma) \leq \Pol(\Psi)$ we will apply \Cref{lem:homo-from-x}. First a definition. For an $l$-ary operation $g$ on a finite set $B$ we say that $i \in [l]$ is an \emph{essential coordinate} if $g$ depends on it, formally, $g(b_1, \dots, b_{i-1}, b_i, b_{i+1}, \dots, b_l) \neq g(b_1, \dots, b_{i-1}, b', b_{i+1}, \dots, b_l)$ for some $b_1, \dots, b_l, b' \in B$. We define $\ess(g) \subseteq [l]$ as the set of all essential coordinates of $g$. 
  Note that if $g$ is idempotent and $|B| \geq 2$, then $\ess(g)$ is nonempty. 
  Also notice that if $g$ is idempotent and $\ess(g) \subseteq \{i\}$, then $g(b_1, \dots, b_l) = b_i$. 
  
  For an $l$-ary $\tuple{f} = (f_P)_{P \in S(\Psi)} \in \Pol(\Gamma)=\Polid(\Gamma)$ we define $X(\tuple{f}) \in C$ as follows.
  If there exists $Q \in S(\Psi)$ such that $\ess(f_P) \subseteq Q$ for every $P \in S(\Psi)$, then we set $X(\tuple{f}) = f_Q(\tuple{d}^Q)$, where $\tuple{d}^Q$ is any tuple in $Q^l$ such that $d^Q_i = i$ for both $i \in Q$ (the value $f_Q(\tuple{d}^Q)$ only depends on these coordinates, so other coordinates of $\tuple{d}^Q$ are irrelevant). There may be more such sets $Q$: this happens only if $\cup_P \,\ess(f_P) = \{i\}$ for some $i \in [l]$ (the union cannot be empty because some $P \in S(\Psi)$ has two elements). But then for any $Q$ containing $\{i\}$, we have $f_Q(\tuple{d}^Q) = i$ by the remark at the end of the last paragraph, so the definition makes sense nevertheless. Finally, if there is no such $Q$ we define $X(\tuple{f}) \in C$ arbitrarily. 

  \Cref{lem:homo-from-x} guarantees that the collection $\xi$ defined by
   $(\xi^{(n)}(\tuple{f}))(\tuple{c}) = X(\tuple{f}^{(\tuple{c})})$ (for $n \in \mathbb{N}$, $\tuple{f} \in \Pol^{(n)}(\Gamma)$, $\tuple{c} \in C^n$) preserves minors. It remains to verify that $g=\xi^{(n)}(\tuple{f})$ preserves every $R \in \Psi$. For this, it is enough to consider values $g(\tuple{c})$ for tuples $\tuple{c}$ containing only elements from  $Q \in S(\Psi)$. But then  each component $f_P$ of $\tuple{f}$ satisfies $\ess(f_P^{(\tuple{c})}) \subseteq Q$ so $g(\tuple{c}) = X(\tuple{f}^{(\tuple{c})}) = f^{(\tuple{c})}_Q(\tuple{d}^Q) = f_Q(\tuple{d}^Q \circ \tuple{c}) = f_Q(\tuple{c})$. That $g$ preserves $R$ thus follows from $\tuple{f}$ preserving $\tilde{R}$, concluding the proof of the first part.

   For the second part, one can define $C$ and $\Psi$ so that the above construction produces $\Gamma'$ essentially equal to $\Gamma$. We omit the simple details. 
\end{proof}

As an immediate consequence of \Cref{thm:main} and \Cref{thm:translation}, we reach the goal of this section.

\begin{theorem} \label{thm:main-small}
  Let $\Psi$ be a set of at most binary relations with small projections on a finite set.  Then $\Pol(\Psi)$ is equivalent to exactly one of the minion cores $\minion{T}$, $\minAk$, \dots listed in \Cref{thm:main}, and all of these minion cores arise in this way.
\end{theorem}

\section{Conclusion} \label{sec:conclusion}

\Cref{fig:order} summarizes our main result. It shows all (up to minion isomorphisms) minion cores of multisorted Boolean clones determined by binary relations and the minion homomorphism order between them (\Cref{thm:main}). The same lattice is obtained by considering clones on finite sets determined by binary relations with small projections (\Cref{thm:main-small}).

\begin{figure}[ht]
\resizebox{12cm}{!}{\includegraphics{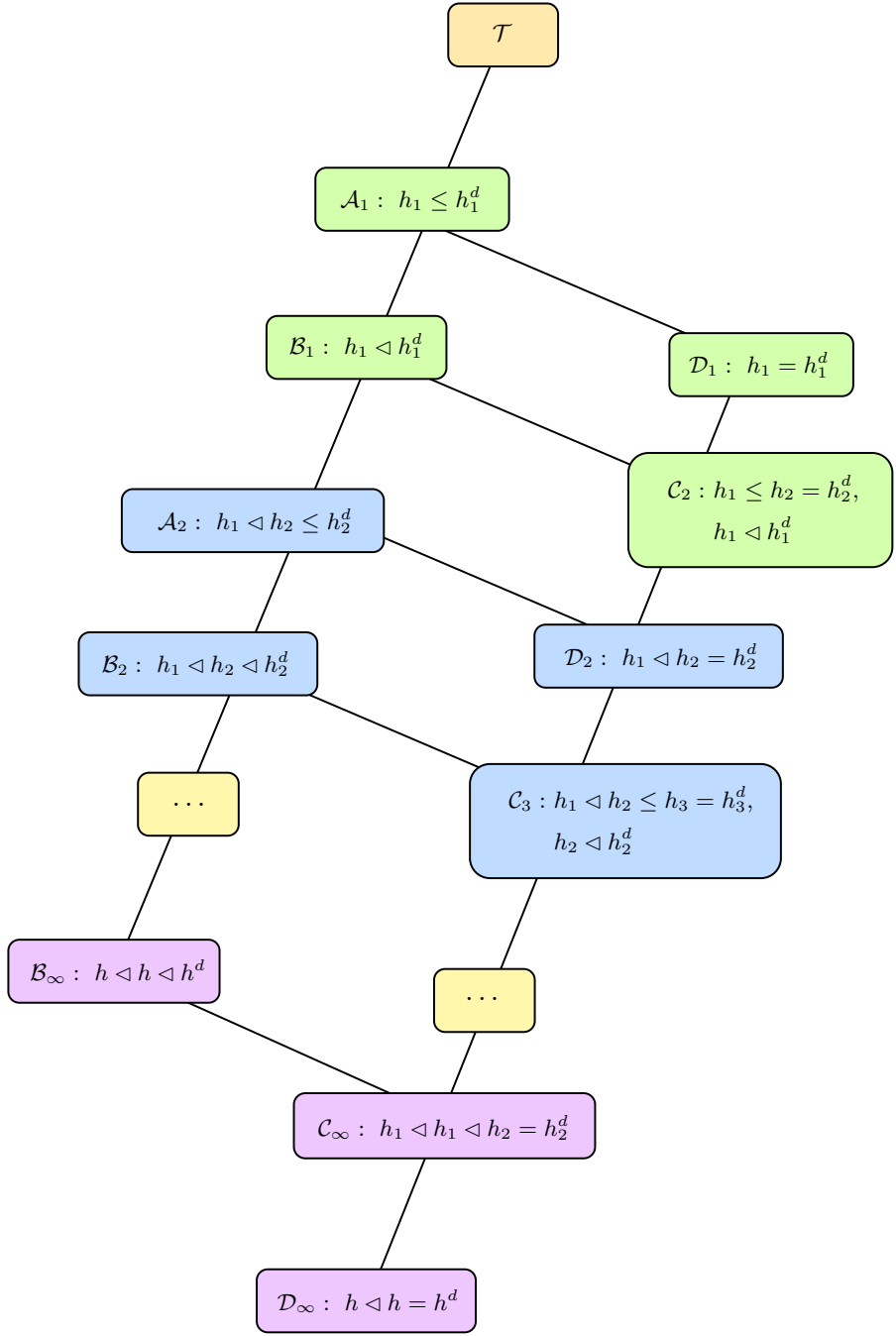}}
\caption{Minion cores of multisorted Boolean clones determined by binary relations.} 
\label{fig:order}
\end{figure}

One direction for future work is to characterize (up to minion homomorphisms) multisorted Boolean \emph{minions} determined by pairs of unary or binary relations; examples of such minions are the minion cores we found. Are there any others? 

A different direction is to stay within multisorted clones but consider relations of arity greater than two. An ongoing work of R. Ol\v s\'ak aims to achieve the characterization in the case that all relations ''live'' on a single sort, that is, types of all relations are constant tuples.

\bibliographystyle{splncs04}
\bibliography{MultiBoolean}

\end{document}